\documentclass[a4paper,10pt,twoside]{article}

\usepackage{amssymb}
\usepackage{amsthm}
\usepackage{amsmath}
\usepackage{enumerate}
\usepackage{hyperref}

\newtheorem{thmA}{Theorem}
\newtheorem{thm}{Theorem}

\newtheorem{lem}{Lemma}

\newtheorem{cor}{Corollary}

\newcommand*{\quart}{\frac14}
\newcommand*{\lb}{\left\{}
\newcommand*{\rb}{\right\}}
\newcommand*{\PP}{\mathbb{P}}
\newcommand*{\di}{\, \mathrm{d} }

\newcommand*{\loc}{\ensuremath{\mathcal{L}}}
\newcommand*{\ind}{\ensuremath{\mathbf{1}_}}

\newcommand*{\curl}{\ensuremath{\mathrm{curl}}}

\begin{document}

\title{Self-intersection local time of planar Brownian motion based on a strong approximation
by random walks}

\author{Tam\'as Szabados\footnote{Address:
Department of Mathematics, Budapest University of Technology and Economics, M\H{u}egyetem rkp. 3, H
\'ep. V em. Budapest, 1521, Hungary, e-mail: szabados@math.bme.hu, telephone: (+36 1)
463-1111/ext. 5907, fax: (+36 1) 463-1677} \\
Budapest University of Technology and Economics}

%\date{}

\maketitle

\bigskip

%%%%%%%%%%%%%%%%%%%%%%%%%%%%%%%%%%%%%%%%%%%%%%%%%%%%%%%%%%%%%%%%%%%

\begin{abstract}
The main purpose of this work is to define planar self-intersection local time by an alternative
approach which is based on an almost sure pathwise approximation of planar Brownian motion by simple,
symmetric random walks. As a result, Brownian self-intersection local time is obtained as an almost
sure limit of local averages of simple random walk self-intersection local times. An important tool is
a discrete version of the Tanaka--Rosen--Yor formula; the continuous version of the formula is
obtained as an almost sure limit of the discrete version. The author hopes that this approach to
self-intersection local time is more transparent and elementary than other existing ones.
\end{abstract}

%%%%%%%%%%%%%%%%%%%%%%%%%%%%%%%%%%%%%%%%%%%%%%%%%%%%%%%%%%%%%%%%%%%

\renewcommand{\thefootnote}{\alph{footnote}}
\footnotetext{ 2010 \emph{MSC.} Primary 60J55. Secondary 60F15, 60G50.} 

%%%%%%%%%%%%%%%%%%%%%%%%%%%%%%%%%%%%%%%%%%%%%%%%%%%%%%%%%%%%%%%%%%%

\section{Introduction} \label{sec:Intro}

Let $(W(t))_{t \ge 0}$ be planar Brownian motion (BM). Formally, its \emph{self-intersection local
time} at the point $x \in \mathbb{R}^2$ up to time $t$ is
\[
\alpha (t, x) = \int_{0}^{t} \int_{0}^{v} \delta (W(v)-W(u)-x) \di u \di v,
\]
where $\delta $ is the Dirac measure at zero. There exist several methods in the literature to make
this definition rigorous. One natural approach, which is the topic of the present work, is to define
$\alpha (t, x)$ as an almost sure limit (when $m \to \infty$) of local averages of self-intersection
local times $\alpha _{m}(t, x)$ of a nested sequence of simple, symmetric planar random walks
$(B_m(t))$. (See the next section for the definition of $B_m$.) For these imbedded random walks
(RW's), self-intersection local time can be defined by elementary counting:
\[
\alpha _{m}(t, x) := 2^{-2m} \#\{(i,j) : 0 \le i \le j < t 2^{2m}, \: B_m(j 2^{-2m}) - B_m(i 2^{-2m})
= x \} .
\]
Then, as it will be seen in Theorem \ref{th:spec_TRY}, the following almost sure limit gives the
Brownian self-intersection local time:
\[
\alpha(t, y) = \lim_{\delta \to 0^+}  \lim_{m \to \infty } \frac{1}{\pi \delta^2} \sum_{ x \in
2^{-m}\mathbb{Z}^2 \cap B_{\delta} (y)} \alpha _m(t, x) \: 2^{-2m}  \qquad (y \ne 0),
\]
where $B_{\delta} (y)$ denotes the disc centered at $y$ with radius $\delta$.

The author hopes that this approach to self-intersection local time is more elementary and more
advantageous from a pedagogical point of view than other existing ones. This method is a special case
of a strong invariance principle for self-inter\-sec\-tion local time. It should be mentioned that
using different methods, strong invariance was shown earlier by Cadre \cite{Cad1997}, and for general
random walks by Bass and Rosen \cite{BR2005}. The method applied in this paper is based on a
Tanaka-like formula, first introduced by Rosen \cite{Ros1986} and then generalized by Yor
\cite{Yor1985}. More exactly, a discrete version of the Tanaka--Rosen--Yor formula for random walks is
given below whose almost sure limit is the continuous version of the formula.

%%%%%%%%%%%%%%%%%%%%%%%%%%%%%%%%%%%%%%%%%%%%%%%%%%%%%%%%%%%%%%%%%%%

\section{Preliminaries} \label{sec:Pre}

A basic tool of the present paper is an elementary construction of Brownian motion. The specific
construction used in the sequel, taken from \cite{Szab1996}, is based on a nested sequence of simple,
symmetric random walks that uniformly converges to the Wiener process (=BM) on bounded intervals with
probability $1$. This will be called \emph{``twist and shrink''} construction. This method is a
modification of the one given by Frank Knight in 1962 \cite{Kni1962}.

We summarize the major steps of the ``twist and shrink'' construction here. We start with an infinite
matrix of independent and identically distributed random variables $X_m(k)$, $\PP \lb X_m(k)= \pm 1
\rb = \frac12$ ($m\ge 0$, $k\ge 1$), defined on the same complete probability space
$(\Omega,\mathcal{F},\PP)$. (All stochastic processes in the sequel will be defined on this
probability space.) Each row of this matrix is a basis of an approximation of the Wiener process with
a dyadic step size $\Delta t=2^{-2m}$ in time and a corresponding step size $\Delta x=2^{-m}$ in
space. Thus we start with a sequence of independent simple, symmetric RW's $S_m(0) = 0$, $S_m(n) =
\sum_{k=1}^{n} X_m(k)$ $(n \ge 1)$.

The second step of the construction is \emph{twisting}. From the independent RW's we want to create
dependent ones so that after shrinking temporal and spatial step sizes, each consecutive RW becomes a
refinement of the previous one.  Since the spatial unit will be halved at each consecutive row, we
define stopping times by $T_m(0)=0$, and for $k\ge 0$,
\[
T_m(k+1)=\min \{n: n>T_m(k), |S_m(n)-S_m(T_m(k))|=2\} \qquad (m\ge 1)
\]
These are the random time instants when a RW visits even integers, different from the previous one.
After shrinking the spatial unit by half, a suitable modification of this RW will visit the same
integers in the same order as the previous RW. We operate here on each point $\omega\in\Omega$ of the
sample space separately, i.e. we fix a sample path of each RW. We define twisted RW's $\tilde{S}_m$
recursively for $k=1,2,\dots$ using $\tilde{S}_{m-1}$, starting with $\tilde{S}_0(n)=S_0(n)$ $(n\ge
0)$ and $\tilde{S}_m(0) = 0$ for any $m \ge 0$. With each fixed $m$ we proceed for $k=0,1,2,\dots$
successively, and for every $n$ in the corresponding bridge, $T_m(k)<n\le T_m(k+1)$. Any bridge is
flipped if its sign differs from the desired:
\[
\tilde{X}_m(n)=\left\{
\begin{array}{rl}
 X_m(n)& \mbox{ if } S_m(T_m(k+1)) - S_m(T_m(k))
= 2\tilde X_{m-1}(k+1), \\
- X_m(n)& \mbox{ otherwise,}
\end{array}
\right.
\]
and then $\tilde{S}_m(n)=\tilde{S}_m(n-1)+\tilde{X}_m(n)$. Then $\tilde{S}_m(n)$ $(n\ge 0)$ is still a
simple symmetric RW \cite[Lemma 1]{Szab1996}. The twisted RW's have the desired refinement property:
\[
\tilde{S}_{m+1}(T_{m+1}(k)) = 2 \tilde{S}_{m}(k) \qquad (m\ge 0, k\ge 0).
\]

The third step of the RW construction is \emph{shrinking}. The sample paths of $\tilde{S}_m(n)$ $(n\ge
0)$ can be extended to continuous functions by linear interpolation, this way one gets
$\tilde{S}_m(t)$ $(t\ge 0)$ for real $t$. The $mth$ \emph{``twist and shrink'' RW} is defined by
\[
\tilde{B}_m(t)=2^{-m}\tilde{S}_m(t2^{2m}).
\]
Then the \emph{refinement property} takes the form
\begin{equation}
\tilde{B}_{m+1}\left(T_{m+1}(k)2^{-2(m+1)}\right) = \tilde{B}_m \left( k2^{-2m}\right) \qquad (m\ge
0,k\ge 0). \label{eq:refin}
\end{equation}
Note that a refinement takes the same dyadic values in the same order as the previous shrunken walk,
but there is a \emph{time lag} in general:
\begin{equation} T_{m+1}(k)2^{-2(m+1)} - k2^{-2m} \ne 0 .
\label{eq:tlag}
\end{equation}

It is clear that this construction is especially useful for local times, since a refinement
approximates the local time of the previous walk, with a geometrically distributed random number of
visits with half-length steps, cf. \cite{SzaSze2005}.

Now let me recall some important facts from \cite{Szab1996} and \cite{SzaSze2005} about the ``twist
and shrink'' construction that will be used in the sequel.

\begin{thmA} \label{th:Wiener}
On bounded intervals the sequence $(\tilde B_m)$ almost surely uniformly converges as $m \to \infty$
and the limit process is Brownian motion $W$. For any $C>1$, and for any $K>0$ and $m \ge 1$ such that
$K 2^{2m} \ge N(C)$, we have
\begin{eqnarray*}
\PP \lb \sup_{0 \le t \le K} |W(t) - \tilde{B}_m(t)| \ge 27 \: C K_*^{\quart} (\log_*K)^{\frac34}
m^{\frac34} 2^{-\frac{m}{2}} \rb
\\
\le \frac{6}{1-4^{1-C}} (K2^{2m})^{1-C} ,
\end{eqnarray*}
where $K_* := K \vee 1$ and $\log_* K := (\log K) \vee 1$.

\end{thmA}
($N(C)$ here and in the sequel denotes a large enough integer depending on $C$, whose value can be
different at each occasion.)

Conversely, with a given Wiener process $W$, one can define the stopping times which yield the
\emph{Skorohod embedded RW's} $B_m(k2^{-2m})$ into $W$. For every $m\ge 0$ let $s_m(0)=0$ and
\begin{equation} \label{eq:Skor1}
s_m(k+1)=\inf{}\{s: s > s_m(k), |W(s)-W(s_m(k))|=2^{-m}\} \qquad (k \ge 0).
\end{equation}
With these stopping times the embedded dyadic walks by definition are
\begin{equation} \label{eq:Skor2}
B_m(k2^{-2m}) = W(s_m(k)) \qquad (m\ge 0, k\ge 0).
\end{equation}
This definition of $B_m$ can be extended to any real $t \ge 0$ by pathwise linear interpolation.

If a Wiener process is built by the ``twist and shrink'' construction described above using a sequence
$(\tilde{B}_m)$ of nested RW's and then one constructs the Skorohod embedded RW's $(B_m)$, it is
natural to ask about their relationship. The next theorem shows that they are asymptotically
equivalent. In general, roughly saying, $(\tilde{B}_m)$ is more useful when someone wants to generate
stochastic processes from scratch, while $(B_m)$ is more advantageous when someone needs discrete
approximations of given processes.

\begin{thmA} \label{th:Wiener_Skor}
For any $C > 1$, and for any $K>0$ and $m \ge 1$ such that $K 2^{2m} \ge N(C)$ we have
\begin{eqnarray*}
\PP \lb \sup_{0\le t\le K} \left| W(t) - B_m(t) \right| \ge 27 \: C K_*^{\quart} (\log_*K)^{\frac34}
m^{\frac34}
2^{-\frac{m}{2}} \rb \\
\le \frac{8}{1 - 4^{1-C}} (K2^{2m})^{1-C} .
\end{eqnarray*}

\end{thmA}

Apply the ``twist and shrink'' construction $d$-times independently, to obtain a $d$-dimensional
Brownian motion $W = (W^1, \dots , W^d)$ (vector components will be denoted by superscripts), the
corresponding Skorohod-embedded RW's $B_m = (B_m^1, \dots ,B_m^d)$, and stopping times $(s_m^1(n),
\dots , s_m^d)$:
\[
B_m^j(k2^{-2m}) = W^j(s_m^j(k)) \qquad (m\ge 0, k\ge 0, j=1, \dots ,d).
\]
Please note that in this paper $d$-dimensional random walks are defined as a vector of $d$ independent
one-dimensional random walks. This means that the coordinate axes of a usual $d$-dimensional random
walk are rotated and the length of a step is multiplied by $\sqrt{d}$.

Then  Theorem \ref{th:Wiener_Skor} and Borel--Canteli lemma imply
\begin{cor} \label{co:strong_d}
Taking a $d$-dimensional Brownian motion $W$ and Skorohod embedded RW's $B_m$, for any $K>0$ and $m
\ge 1$ one has
\begin{multline} \label{eq:strong_d}
\sup_{0\le t\le K} \left| W(t) - B_m(t) \right| = O\left((\log n)^{\frac34}\; n^{-\frac14}\right) =
O\left(m^{\frac34} 2^{-\frac{m}{2}} \right) \quad \text{a.s.} ,
\end{multline}
where $n = K 2^{2m}$ denotes the number of vertices of an imbedded random walk $B_m$ over the time
interval $[0, K]$.
\end{cor}

As it was mentioned above, the above approach to Brownian motion is especially suitable to give an
elementary definition of Brownian local time as an a.s. limit of RW local times, cf.
\cite{SzaSze2005}. In fact, this was the main motivation to find a similar definition of planar
self-intersection local time as well.

The idea that random walk approximations can be applied to obtain results about Brownian local time
goes back to Knight \cite{Kni1963}, who proved the celebrated Ray-Knight theory this way. R\'ev\'esz
\cite{Rev1981} and Cs\'aki \& R\'ev\'esz \cite{CsR1983} were the first to prove strong invariance for
local times.

Let $\tilde{\ell}_m(0,x) := 0$ and
\begin{equation}\label{eq:dloct}
\tilde{\ell}_m(k,x) := \#\{j : 0 \le j < k, \: \tilde{S}_m(j) = x \} .
\end{equation}
Define \emph{local times} of the ``twist and shrink'' RW $\tilde B_m$ or imbedded RW's $B_m$ by
\begin{equation}\label{eq:tasloct}
\tilde{\loc}_m(t,x) := 2^{-m} \tilde{\ell}_m \left(t2^{2m}, x 2^m \right).
\end{equation}
Then \cite{SzaSze2005} shows
\begin{thmA}
On any strip $[0, K] \times \mathbb{R}$,
\[
\lim_{m \to \infty} \tilde{\loc}_m(t,x) =  \loc(t,x) \quad \text{a.s.},
\]
uniformly in $(t,x)$, where $\loc(t,x)$ is the local time of BM. Hence this automatically gives a
version of Brownian local time  which is continuous in $(t,x)$.

Moreover, for all $K>0$ and $m \ge 1$ one has
\[
\sup_{(t,x) \in [0,K] \times \mathbb{R}} \left|\loc(t,x) - \tilde{\loc}_{m}(t,x)\right| = O\left((\log
n)^{\frac34}\; n^{-\frac14}\right)  \quad \text{a.s.} ,
\]
where $n = K 2^{2m}$.
\end{thmA}
Similar statements hold for $\loc_m(t,x)$ computed from Skorohod embedded RW's $B_m$ as well.
Interestingly, the rate of convergence is the same for local time as for the approximation of Brownian
motion with the ``twist and shrink'' construction. While this rate is much weaker than the optimal
Koml\'os--Major--Tusn\'ady rate $(\log n) \; n^{-\frac12}$ in the case of BM, it just slightly differs
from the optimal $(\log n)^{\frac12} \; (\log\log n)^{\frac14} \; n^{-\frac14}$ in the case of local
time.

%%%%%%%%%%%%%%%%%%%%%%%%%%%%%%%%%%%%%%%%%%%%%%%%%%%%%%%%%%%%%%%%%%%

\section{A discrete It\^o's formula}
\label{sec:disc_Ito}

It is interesting that one can give discrete versions of It\^o's formula and of It\^o--Tanaka--Meyer
formula, which are purely algebraic identities, not assigning any probabilities to the terms. Despite
this, the usual It\^o's formula follows fairly easily in a proper probability setting.

Discrete It\^o formulas are not new. Apparently, the first such formula was given by Kudzma in 1982
\cite{Kud1982}. The elementary algebraic approach used in the present paper is different from that; it
was introduced by the author in 1989 \cite{Szab1990}.

First we need definitions of discrete line integrals and conservative vector fields on a grid. Fix an
initial point $a \in \mathbb{R}^d$ and step-size (mesh) $h > 0$. Consider the grid $\mathcal{G}(a,h)
:= a + h \mathbb{Z}^d$, and let $f = (f^1, \dots, f^d) : \mathcal{G}(a,h) \rightarrow \mathbb{R}^d$ be
a vector field on this grid. (Coordinates of a vector will always be denoted by superscripts.) Take an
arbitrary broken line (\emph{a discrete path}) $\gamma $ that goes through finitely many (not
necessarily distinct) oriented edges between adjoining vertices of the grid. A typical such edge is
$[x, x + \mu h e_j]$, where $x \in \mathcal{G}(a,h)$, $e_j$ $(1 \le j \le d)$ is a coordinate unit
vector and $\mu = \pm 1$. (The order of the two vertices is important!) A discrete path $\gamma $ is a
formal sum of such oriented edges (that is, a 1-chain):
\[
\gamma = \sum_{r=1}^n [x_r, x_r + \mu_r h e_{j_r}] \qquad (1 \le j_r \le d).
\]

By definition, the corresponding \emph{discrete path integral} (or \emph{trapezoidal sum}) of $f$ over
$\gamma $ is defined as
\[
T_{\gamma } \left(f, \gamma'_{0} \right) h := \frac{h}{2} \sum_{r=1}^n \mu_r  \left(f^{j_r}(x_r) +
f^{j_r}(x_r + \mu_r h e_{j_r})\right) .
\]
Here the symbol $\gamma '_0$ refers to the unit tangents of $\gamma $ along edges, and $(f,
\gamma'_0)$ denotes dot product. When it will be convenient, $T_{\gamma } \left(f(x), \-
\gamma'_{0}(x) \right) h$ will be written to show the dummy variable of the summation.

The ordinary path (line) integral of a vector field $f$ over a path $\gamma $ will be denoted by
$\int_{\gamma } (f, \gamma'_0) \di s$ or by $\int_{\gamma } f(x) \cdot \di x$, where $\di s$ refers to
the length element of $\gamma $.

The above definition of discrete path integral shows that the orientation of an edge is defined by the
order of its two vertices: it is positive if the edge goes increasingly in a coordinate and negative
in the opposite case. If $\gamma = \emptyset$, we define $T_{\gamma } \left(f, \gamma '_0 \right) h =
0$.

A vector field $f$ is called \emph{discrete conservative} on the grid $\mathcal{G}(a,h)$ if for any
$b, c \in \mathcal{G}(a,h)$ and for any discrete path $\gamma$ going from $b$ to $c$ through edges of
adjoining neighbor vertices of $\mathcal{G}(a,h)$, the discrete path integral does not depend on the
path $\gamma $, it depends only on the initial point $b$ and endpoint $c$. In this case the notation
$T_{b}^c \left(f, \gamma '_0 \right) h$ will be used for the trapezoidal sum. Clearly, $f$ is discrete
conservative if and only if $T_{b}^c \left(f, \gamma '_0 \right) h = 0$ whenever $b = c$, that is, the
path $\gamma $ is closed.

When $f$ is discrete conservative, one can define a \emph{discrete potential} $g: \mathcal{G}(a,h)
\rightarrow \mathbb{R}$ by the formula $g(x) := T_{a}^x \left(f, \gamma '_0 \right) h$. Then for any
points $b$ and $c$ in the grid, and for any discrete path $\gamma $ connecting them, one has
$T_{\gamma } \left(f, \gamma'_{0} \right) h = g(c) - g(b)$.

The following discrete It\^o's formula (which is a simple algebraic identity) already appeared in
\cite[Section 5]{Szab1990} in the two-dimensional case. It is based on the principle that though our
random walk is ``diagonal'', constructed from independent one-dimensional random walks, the discrete
integrals below go ``parallel to the coordinate axes'', as in a standard continuous It\^o's formula.
Also, the main object in our discrete formula is the ``integrand'' $f$ in the stochastic sum, which
corresponds to the gradient of a scalar field in a standard continuous It\^o's formula. That may
explain why we suppose that $f$ be discrete conservative. These methods make it convenient to deduce
important continuous formulae from the discrete ones.

\begin{lem}\label{le:discrete_Ito}
Take $a \in \mathbb{R}^d$, step $h > 0$, and a discrete conservative time-dependent vector field $f =
(f^1, \dots , f^d) : h^2 \mathbb{Z}_+ \times \mathcal{G}(a,h) \rightarrow \mathbb{R}^d$. Consider a
sequence $X_r = (X^1_{r}, \dots , X^d_{r})$ ($r \ge 1$), where $X^j_{r} = \pm 1$. Define partial sums
$S_0=a $, $S_n = a+h(X_1+\cdots +X_n)$ ($n \ge 1$) and discrete time instants $t_r = r h^2$ ($0 \le r
\le n$). Assume that the steps of $(S_n)$ are performed in time steps $h^2$. Then the following
equalities hold:
\begin{multline} \label{eq:disc_Strat}
T_{x=S_0}^{S_n} \left(f(t_n,x), \gamma '_0(x) \right) h \\
= \sum_{r=1}^{n}  T_{x=S_0}^{S_r} \left(\left\{f(t_r,x) - f(t_{r-1},x)\right\}, \gamma
'_0(x) \right) h   \\
+ \sum_{r=1}^{n} \sum_{j=1}^{d} \frac{f^j \left(t_{r-1}, S_{r-1} + \sum_{i=1}^{j-1} h X_{r}^{i} e_i
\right) + f^j \left(t_{r-1}, S_{r-1} + \sum_{i=1}^{j} h X_{r}^{i} e_i \right)}{2} h X_r^j
\end{multline}
(discrete Stratonovich formula). Alternatively,
\begin{multline} \label{eq:disc_Ito}
T_{x=S_0}^{S_n} \left(f(t_n,x), \gamma '_0(x) \right) h \\
= \sum_{r=1}^{n}  T_{x=S_0}^{S_r} \left(\left\{f(t_r,x) - f(t_{r-1},x)\right\}, \gamma
'_0(x) \right) h   \\
+ \sum_{r=1}^{n} \sum_{j=1}^{d} f^j \left(t_{r-1}, S_{r-1} +
\sum_{i=1}^{j-1} h X_{r}^{i} e_i \right) h X_r^j  \\
+  \frac12 \sum_{r=1}^n \sum_{j=1}^{d} \frac{f^j \left(t_{r-1}, S_{r-1} + \sum_{i=1}^{j} h X_{r}^{i}
e_i \right) - f^j \left(t_{r-1}, S_{r-1} + \sum_{i=1}^{j-1} h X_{r}^{i} e_i \right)} {h X_r^j} h^2
\end{multline}
(discrete It\^o's formula).
\end{lem}

\begin{proof}
Algebraically,
\begin{eqnarray*}
\lefteqn{T_{x=S_0}^{S_r} \left(f(t_r,x), \gamma '_0(x) \right) h
- T_{x=S_0}^{S_{r-1}} \left(f(t_{r-1},x), \gamma '_0(x) \right) h} \\
& = & T_{x=S_0}^{S_r} \left(f(t_r,x), \gamma '_0(x) \right) h -
T_{x=S_0}^{S_{r}} \left(f(t_{r-1},x), \gamma '_0(x) \right) h \\
& + & T_{x=S_0}^{S_r} \left(f(t_{r-1},x), \gamma '_0(x) \right) h - T_{x=S_0}^{S_{r-1}}
\left(f(t_{r-1},x), \gamma '_0(x) \right) h .
\end{eqnarray*}
Using the assumption that $f$ is discrete conservative, we get that
\begin{eqnarray*}
\lefteqn{T_{x=S_0}^{S_r} \left(f(t_{r-1},x), \gamma '_0(x) \right) h - T_{x=S_0}^{S_{r-1}}
\left(f(t_{r-1},x), \gamma '_0(x) \right)
h} \\
& = & T_{x=S_{r-1}}^{S_r} \left(f(t_{r-1},x), \gamma '_0(x) \right) h
\end{eqnarray*}
and
\begin{eqnarray} \label{eq:disc_Ito_proof}
\lefteqn{T_{x=S_{r-1}}^{S_r} \left(f(t_{r-1},x), \gamma '_0(x) \right) h
} \\
& = & \sum_{j=1}^{d} \frac{f^j \left(t_{r-1}, S_{r-1} + \sum_{i=1}^{j-1} h X_{r}^{i} e_i \right) + f^j
\left(t_{r-1}, S_{r-1} + \sum_{i=1}^{j} h X_{r}^{i} e_i \right)}{2} h X_r^j
\nonumber \\
& = & \sum_{j=1}^{d} f^j \left(t_{r-1}, S_{r-1} +
\sum_{i=1}^{j-1} h X_{r}^{i} e_i \right) h X_r^j \nonumber \\
& + &  \frac12 \sum_{j=1}^{d} \frac{f^j \left(t_{r-1}, S_{r-1} + \sum_{i=1}^{j} h X_{r}^{i} e_i
\right) - f^j \left(t_{r-1}, S_{r-1} + \sum_{i=1}^{j-1} h X_{r}^{i} e_i \right)} {h X_r^j} h^2 .
\nonumber
\end{eqnarray}
The first equality follows from the fact that if $X_r^j=1$, one has a positively oriented edge, while
if $X_r^j=-1$, one has a negatively oriented edge in the trapezoidal sum. Then the second equality
follows since $1/X_r^j = X_r^j$. Summing up for $r=1, \dots, n$, the sum on the left telescopes, and
from the two equalities one obtains the two formulae, respectively.
\end{proof}

One can introduce \emph{partial local times} $\loc^{\mu}_h (t_n,x)$ ($n \ge 0$) of the series $(S_n)$
with spatial step $h>0$, time step $h^2$, and $\mu \in \{1, -1 \}^d$: $\loc^{\mu}_h(0,x) := 0$ and
\begin{equation}\label{eq:DPLT}
\loc^{\mu}_h(t_n,x) := h \: \#\{j : 0 \le j < n, \: S_{j} = x, \: S_{j+1} = x + h \mu  \} ,
\end{equation}
where $n \ge 1$ and $x \in \mathcal{G}(a,h)$. The (total) \emph{local time} is
\begin{equation}\label{eq:DLT}
\loc_h(t_n,x) := \sum_{\mu \in \{1, -1 \}^d} \loc^{\mu }_h(t_n,x) = h \: \#\{j : 0 \le j < n, \: S_{j}
= x\}.
\end{equation}
Here our convention differs from the usual one: time $0$ is counted, but time $n$ is not. The reason
is that this better fits the discrete formula below.

\begin{lem}\label{le:disc_Ito_Tan_Mey}
With the same assumptions as above in Lemma \ref{le:discrete_Ito}, except that the vector field $f$
does not depend on time, $f : \mathcal{G}(a,h) \rightarrow \mathbb{R}^d$, one also has
\begin{eqnarray} \label{eq:disc_Ito_Tan_Mey}
\lefteqn{T_{x=S_0}^{S_n} \left(f(x), \gamma '_0(x) \right) h}
\\
& = & \sum_{r=1}^{n} \sum_{j=1}^{d} f^j \left(S_{r-1} +
\sum_{i=1}^{j-1} h X_{r}^{i} e_i \right) h X_r^j \nonumber \\
& + & \frac{1}{2} \sum_{x \in a + h \mathbb{Z}^d} \sum_{\mu
\in \{1, -1\}^d} \loc^{\mu}_h(t_n, x)  \nonumber \\
& & \times \sum_{j=1}^{d} \mu ^j \left\{ f^j \left(x + \sum_{i=1}^{j} h \mu ^i e_i \right) - f^j
\left(x + \sum_{i=1}^{j-1} h \mu ^i e_i \right)\right\} \nonumber
\end{eqnarray}
(discrete It\^o--Tanaka--Meyer formula).
\end{lem}

\begin{proof}
Continuing  (\ref{eq:disc_Ito_proof}) in the proof of the previous lemma,
\begin{eqnarray*}
\lefteqn{T_{x=S_{r-1}}^{S_r} \left(f(x), \gamma '_0(x) \right) h
} \\
& = & \sum_{j=1}^{d} f^j \left(S_{r-1} + \sum_{i=1}^{j-1} h X_{r}^{i}
e_i \right) h X_r^j  \\
& + &  \frac{1}{2} \sum_{x \in a + h \mathbb{Z}^d} \sum_{\mu
\in \{1, -1\}^d} h \ind{\{S_{r-1}=x, S_r=x+h \mu \}} \\
& & \times \sum_{j=1}^{d} \mu ^j \left\{ f^j \left(x + \sum_{i=1}^{j} h \mu ^i e_i \right) - f^j
\left(x + \sum_{i=1}^{j-1} h \mu ^i e_i \right) \right\} .
\end{eqnarray*}
Again, summing up for $r=1, \dots, n$, the sum on the left telescopes, and on the right one obtains
the asserted formula.
\end{proof}

%%%%%%%%%%%%%%%%%%%%%%%%%%%%%%%%%%%%%%%%%%%%%%%%%%%%%%%%%%%%%%%%%%%

\section{Constructing discrete conservative vector \\ fields on a
planar grid} \label{sec:grid}

For the sake of simplicity, from now on only the planar case ($d=2$) will be discussed, as this is the
case that will be used in the sequel. The problem that we consider in this section is that given a
differentiable scalar field $g$ in the plane, its gradient $\nabla g$ is \emph{not} a discrete
conservative vector field on a grid $\mathcal{G}(a,h) = a + h \mathbb{Z}^2$ in general. We want to
modify $\nabla g$ so that the resulting vector field $f$ be discrete conservative on the grid, but
still do not differ much from $\nabla g$.

Let us call any $E := [x, x+he_1] \times [x, x+he_2]$, $x \in \mathcal{G}(a,h)$ an \emph{elementary
rectangle of the grid}. It is clear by the previous definitions that a vector field $f$ is discrete
conservative on $\mathcal{G}(a,h)$ if and only if for the counterclockwise directed boundary $\gamma =
\partial E$ of any elementary rectangle, the \emph{discrete curl} of $f$:
\begin{eqnarray} \label{eq:curl}
\lefteqn{(\curl_h f)(x) := \frac{1}{h^2} T_{\gamma } (f, \gamma'_0)
h} \\
& = & \frac{1}{2h} \left\{f^1(x^1,x^2) +
f^1(x^1+h,x^2) + f^2(x^1+h,x^2) \right.  \nonumber \\
& + &  f^2(x^1+h,x^2+h) -  f^1(x^1+h,x^2+h) \nonumber \\
& - & \left. f^1(x^1,x^2+h) - f^2(x^1,x^2+h) - f^2(x^1,x^2) \right\} \nonumber
\end{eqnarray}
is zero. Observe that a discrete curl is a trapezoidal sum over the edges of an elementary rectangle,
divided by $h^2$, the area of the rectangle.

Starting with a scalar field $g \in C^3(\mathbb{R}^2)$, we introduce the following modification
algorithm to obtain a discrete conservative vector field $f$ on $\mathcal{G}(a,h)$. By translation, we
may assume that $a=0$. First we set $f^j(x^1,x^2) = (D_j g)(x^1,x^2)$ whenever $x^1=0$ or $x^2=0$.
($D_j$ denotes partial differentiation with respect to $x^j$.)

Let us consider now elementary rectangles of the grid in the first quadrant. We proceed inductively
with layers of rectangles whose lower left (SW) vertex is $(x^1,x^2)$, $x^1 \wedge x^2 = r h$,
$r=0,1,2,\dots$. Because of symmetry, it is enough to describe the algorithm when $x^2 \le x^1$. In
the $r$th layer we proceed as $x^1 = jh$, $j=0,1,2,\dots$. The next rectangle inherits the values of
$f$ defined on the vertices of previous rectangles, except for the upper right (NE) vertex, which is
called the \emph{``new'' vertex}. For this, compute the \emph{modified discrete curl}
\begin{eqnarray*}
\lefteqn{(\curl_h^g f)(x) := \frac{1}{2h} \left\{f^1(x^1,x^2) +
f^1(x^1+h,x^2) + f^2(x^1+h,x^2)  \right. } \\
& + & (D_2g)(x^1+h,x^2+h) - (D_1g)(x^1+h,x^2+h) \\
& - & \left. f^1(x^1,x^2+h) - f^2(x^1,x^2+h) - f^2(x^1,x^2) \right\},
\end{eqnarray*}
(which is not zero in general) and for $j=1,2$ set
\begin{equation}\label{eq:field_modif}
f^j(x^1+h,x^2+h) = (D_jg)(x^1+h,x^2+h) + (-1)^{j-1} h \: (\curl_h^g f)(x)
\end{equation}
at the ``new'' vertex. It is clear that the so defined $f$ is discrete conservative in the first
quadrant. Observe that the two modification terms $(-1)^{j-1} h \: (\curl_h^g f)(x)$ have the same
absolute value, and in $(\curl_h f)(x)$ they both have minus sign.

In other quadrants the situation is analogous to the case of the first quadrant, but the ``new''
vertex is the upper left in the second, the lower left in the third, and the lower right in the fourth
quadrant. The signs of the modification terms $\pm h \: (\curl_h^g f)(x)$ in (\ref{eq:field_modif})
have to be changed accordingly as well.

It remains to see how large the difference between $f$ and $\nabla g$ is. For the sake of simplicity,
consider only points in the first quadrant, points in other quadrant being analogous. We claim that
these errors accumulate only diagonally.

In fact, if one considers two neighbor rectangles with a common edge, then we can see from
(\ref{eq:curl}) and (\ref{eq:field_modif}) that the discrete curl of the ``new'' rectangle (which is
right or up from the ``old'') does not inherit the modification terms $\pm h \: (\curl_h^g f)(x)$ of
the ``old'' rectangle. The reason is that the modification terms of the ``old'' rectangle cancel in
the ``new'' curl; out of the two pairs of edges joining at the NE vertex of the ``old'' rectangle, one
pair of parallel edges has opposite directions, so the sign of the modification term changes, while
the other pair has identical directions, so the sign of the modification term remains the same. On the
other hand, in the case of two rectangles with a single common vertex (so which are in diagonal
position), the curl of the ``new'' (NE) rectangle does inherit the modification terms of the ``old''
(SW) rectangle, because both pairs of parallel edges have opposite directions.

Thus (\ref{eq:field_modif}) implies that for any $n \ge 1$ and $R
> 0$,
\begin{eqnarray} \label{eq:error_prop}
\lefteqn{\sup_{|x^1| \wedge |x^2| = nh; |x| \le R} |f^j(x) -
(D_jg)(x)|} \\
& \le & nh \sup_{|x| \le R} |(\curl_h \nabla g)(x)| \le \frac{R}{\sqrt{2}} \sup_{|x| \le R} |(\curl_h
\nabla g)(x)| . \nonumber
\end{eqnarray}

Thus the error estimation reduces to an estimate between the ``true'' $\curl \nabla g$ $:= D_{12} g -
D_{21}g = 0$ and the discrete $\curl_h \nabla g$. Or, more precisely, between the ``true'' path
integral $\int_{\gamma } (\nabla g, \gamma '_0) ds = 0$ and the discrete path integral $T_{\gamma }
(\nabla g, \gamma '_0) h$ of the conservative vector field $\nabla g$ over the boundary $\gamma $ of
an elementary rectangle.

Now, as it is well-known, if $\phi \in C^2(\mathbb{R})$, the error between the integral and the
trapezoidal area of $\phi$ on $[x,x+h]$ is
\begin{equation}\label{eq:trapez_error}
\int_{x}^{x+h} \phi(u) du - h \frac{\phi(x)+\phi(x+h)}{2} = -\frac{h^3}{12} \phi''(x+sh),
\end{equation}
where $0 \le s \le 1$.

When $E := [x, x+he_1] \times [x, x+he_2]$ is an elementary rectangle and $\gamma  = \partial E$, it
follows that for any $g \in C^3(\mathbb{R}^2)$ and for any $x \in \mathbb{R}^2$,
\begin{eqnarray*}
(\curl_h \nabla g)(x) & = & \frac{1}{h^2} T_{\gamma } (\nabla g,
\gamma'_0) h \\
& = & \frac{1}{h^2} \left\{ T_{\gamma } (\nabla g, \gamma'_0) h -
\int_{\gamma } (\nabla g, \gamma'_0) \di s \right\} \\
& = & \frac{h}{12} \left\{ (D^3_{1}g)(x^1+s_1h,x^2) +
(D^3_{2}g)(x^1+h,x^2+s_2h) \right. \\
& & - \left. (D^3_{1}g)(x^1+s_3h,x^2+h) - (D^3_{2}g)(x^1,x^2+s_4h) \right\} ,
\end{eqnarray*}
where $0 \le s_j \le 1$. This implies that
\begin{equation} \label{eq:error_curl}
|(\curl_h \nabla g)(x)| \le \frac{1}{6} h \: \epsilon_E,
\end{equation}
where $\epsilon_E := \sup_{x, y \in E} \{|(D^3_{1}g)(x) - (D^3_{1}g)(y)|, |(D^3_{2}g)(x) -
(D^3_{2}g)(y)|\}$, which goes to zero as $h \to 0^+$.

Combining (\ref{eq:error_prop}) and (\ref{eq:error_curl}), we obtain the following
\begin{lem}\label{le:error_in_field}
Let $a \in \mathbb{R}^2$, $h > 0$, $R > 0$, and $g \in C^3(\mathbb{R}^2)$. Define $\epsilon(h) =
\epsilon_g(h,a,R)$ by
\begin{equation}\label{eq:epsh}
\epsilon(h) := \sup \{|(D^3_{j}g)(x) - (D^3_{j}g)(y)| : |x-a|, |y-a| \le R+h; |x-y| \le h\sqrt{2};
j=1,2 \}
\end{equation}
(which goes to zero as $h \to 0^+$). Let $f$ denote the discrete conservative vector field on the grid
$\mathcal{G}(a,h) = a + h \mathbb{Z}^2$, obtained from $\nabla g$ by the modification algorithm
described above. Then
\[
\sup_{|x-a| \le R} |f(x) - (\nabla g)(x)| \le \frac{R}{6} \: h \: \epsilon(h) .
\]
\end{lem}

This lemma expresses the fact that the error of the above modification algorithm even when divided by
$h$ can be made uniformly arbitrary small on any bounded planar set by choosing a small enough $h$.

%%%%%%%%%%%%%%%%%%%%%%%%%%%%%%%%%%%%%%%%%%%%%%%%%%%%%%%%%%%%%%%%%%%

\section{Planar It\^o's formula as an almost sure limit of the
discrete formula} \label{sec:Ito}

Let us apply now the planar ($d=2$) case of the discrete It\^o's formula (\ref{eq:disc_Ito}) to a
random, time-dependent scalar field $g : \Omega \times \mathbb{R}_+ \times \mathbb{R}^2 \to
\mathbb{R}$, $g(\omega,t,x)$, which is measurable in $\omega$ for all $(t,x)$, and is $C^{1,3}$ in
$(t,x)$ for almost all $\omega $.

More exactly, fixing $a \in \mathbb{R}^2$ and taking $m=0,1,2,\dots$ and $h=2^{-m}$, construct first a
sequence of discrete conservative vector fields $f_m(\omega, t, x)$ on the grids
$\mathcal{G}(a,2^{-m})$ from $(\nabla g)(\omega, t, x)$, for each $\omega $ and $t$ fixed, by the
modification algorithm discussed in the previous section. (In this paper $\nabla g = (D_1 g, D_2 g)$
denotes gradient of $g(\omega , t, x)$ with respect to $x$ and $D_t g$ its derivative with respect to
$t$.) Fixing a bounded time interval $[0, K]$, by a slight generalization of Lemma
\ref{le:error_in_field}, for any $\omega $ fixed we get that
\begin{equation}\label{eq:fm_g}
\sup_{t \in[0, K]} \sup_{|x-a| \le R}  |f_m(\omega, t, x) - (\nabla g)(\omega, t, x)| \le \frac{R}{6}
\: 2^{-m} \: \epsilon_K^0(2^{-m}) ,
\end{equation}
where
\[
\epsilon_K^0(h) := \sup \{|(D^3_{j}g)(\omega,t,x) - (D^3_{j}g)(\omega,t,y)| \}
\]
and the supremum is taken for all $ |x-a|, |y-a| \le R+h, |x-y| \le h\sqrt{2}$, $t \in [0, K]$ and
$j=1,2$. Then $\epsilon_K^0(2^{-m}) \to 0$ as $m \to \infty$.

Moreover, it is clear from the modification algorithm (\ref{eq:field_modif}) that at any point of the
grid, $f_m$ differs from $\nabla g$ by a finite linear combination of $D_j g$ values, so $f_m(\omega,
t, x)$ is continuously differentiable as a function of $t$, like $\nabla g$. Moreover, by taking
derivative of (\ref{eq:field_modif}) with respect to $t$, we obtain
\begin{multline*}
(D_t f_m^j)(\omega, t, x^1+h,x^2+h) = (D_j D_t g)(\omega, t, x^1+h,x^2+h) \\
+ (-1)^{j-1} h \: (\curl_h^{D_t g} (D_t f_m))(x) .
\end{multline*}
By the same argument that lead to Lemma \ref{le:error_in_field}, for any fixed $\omega$ it follows
that
\begin{equation}\label{eq:field_modif_t}
\sup_{t \in  [0, K]} \sup_{|x-a| \le R}  |(D_t f_m)(\omega, t, x) - (D_t \nabla g)(\omega, t, x)| \le
\frac{R}{6} \: h \: \epsilon_K^1(h) ,
\end{equation}
where
\[
\epsilon_K^1(h) := \sup \{|(D^3_{j} D_t g)(\omega, t, x) - (D^3_{j} D_t g)(\omega, t, y)|  \}
\]
and the supremum is taken for all $|x-a|, |y-a| \le R+h, |x-y| \le h\sqrt{2}$, $t \in [0, K]$ and
$j=1,2$. Then $\epsilon_K^1(h) \to 0$ as $h \to 0^+$.

Second, start with a planar Brownian motion $(W(t))_{t \in \mathbb{R}_+}$ constructed as in Section
\ref{sec:Pre}, but shifted so that $W(0)=a$. Then take the planar Skorohod embedded random walks
$(B_m(t))_{t \in \mathbb{R}_+}$, $B_m^j(r2^{-2m})=W^j(s_m^j(r))$ ($j=1,2$) in (\ref{eq:disc_Ito}).
That is, let $S_r:=B_m(r2^{-2m})$ and
\[
X_r = X_m(r) := 2^m \left\{B_m(r2^{-2m}) - B_m((r-1)2^{-2m}) \right\} .
\]
Then $(X_m(r))_{r=1}^{\infty}$ is a two-dimensional, independent, $(\pm 1, \pm 1)$ symmetric coin
tossing sequence. Define \emph{stochastic sums} by
\begin{multline} \label{eq:stocsum1}
\left(f_m(\omega, u, W) \cdot W \right)^m_t := \sum_{r=1}^{n} \left\{f_m^1\left(\omega , t_{r-1},
B_m^1(t_{r-1}), B_m^2(t_{r-1})\right) \: 2^{-m}
X_m^1(r) \right. \\
\left. + f_m^2\left(\omega , t_{r-1}, B_m^1(t_r), B_m^2(t_{r-1}) \right)\: 2^{-m} X_m^2(r) \right\},
\end{multline}
where $t_r := r 2^{-2m}$ and $n := \lfloor t 2^{2m} \rfloor$. (Of course, $B_m$, $X_m$, and $W$ all
depend on $\omega $, but this dependence is not shown here and below, to simplify the notation.)

Now the discrete It\^o's formula (\ref{eq:disc_Ito}) can be written as
\begin{eqnarray} \label{eq:disc_Ito_planar}
\lefteqn{T_{x=a}^{B_m(t_{n})} \left(f_m(\omega , t_{n}, x), \gamma '_0(x) \right) 2^{-m}}
\\
& = & \sum_{r=1}^{n} T_{x=a}^{B_m(t_r)} \left(\left\{f_m(\omega , t_r,x)
- f_m(\omega , t_{r-1}, x)\right\}, \gamma'_0(x) \right) 2^{-m}  \nonumber \\
& + & \left(f_m(\omega, u, W) \cdot W \right)^m_t \nonumber \\
& + &  \frac12 \sum_{r=1}^{n} \left\{ \frac{f_m^1 \left(\omega , t_{r-1}, B_m^1(t_r), B_m^2(t_{r-1})
\right) - f_m^1 \left(\omega , t_{r-1}, B_m(t_{r-1}) \right)}
{2^{-m} X^1_m(r)} \right. \nonumber \\
& & \qquad + \left. \frac{f_m^2 \left(\omega, t_{r-1}, B_m(t_r) \right) - f_m^2 \left(\omega, t_{r-1},
B_m^1(t_r), B_m^2(t_{r-1}) \right)} {2^{-m} X^2_m(r)}\right\} 2^{-2m} . \nonumber
\end{eqnarray}

Our strategy is that we show that each term in this formula, except for the stochastic sum, almost
surely uniformly converges to the corresponding term of the planar It\^o's formula, on any bounded
time interval. Then it follows that the stochastic sum almost surely uniformly converges as well (to
the stochastic integral), on any bounded time interval. Hence at the same time we obtain a proof of
the planar It\^o formula as an almost sure uniform limit of the discrete formula.

\begin{thm} \label{th:planar_Ito}
Suppose $g(\omega,t,x)$ is measurable in $\omega$ for all $(t,x)$, and is $C^{1,3}$ in $(t,x)$ for
almost every $\omega $. For each $m=0,1,2,\dots$ and for all $\omega $ and $t$, let $f_m(\omega, t,
x)$ denote the discrete conservative modification of $(\nabla g)(\omega, t, x)$ on the grid
$\mathcal{G}(a,2^{-m})$. Taking a planar Brownian motion $W$, for each $m$ define the Skorohod
embedded random walk $B_m$. Then for arbitrary $K>0$,
\[
\sup_{t\in[0,K]}\left|\left(f_m(\omega, u, W) \cdot W \right)_t^m - \int_0^t (\nabla g)(\omega , u,
W(u)) \cdot \di W(u) \right| \rightarrow 0
\]
almost surely as $m \to \infty$, and for any $t \ge 0$ we obtain the planar It\^o's formula as an
almost sure limit of the discrete formula (\ref{eq:disc_Ito_planar}):
\begin{eqnarray} \label{eq:Ito_form}
\lefteqn{g(\omega , t, W(t)) - g(\omega , 0, W(0)) = \int_{0}^{t}
(D_t g)(\omega , u, W(u)) \di u} \\
& & + \int_{0}^{t} (\nabla g)(\omega , u, W(u)) \cdot \di W(u) + \frac12 \int_{0}^{t} (\Delta
g)(\omega , u, W(u)) \di u . \nonumber
\end{eqnarray}
\end{thm}

\begin{proof}
We are going to prove (\ref{eq:Ito_form}) pathwise. For this, let $\Omega_0$, $\PP \lb \Omega_0 \rb =
1$, denote a subset of the sample space $\Omega $, on which, as $m \to \infty$, $B_m$ uniformly
converges to $W$ on $[0,K]$ and $g(\omega , t, x)$ is $C^{1,3}$ as a function of $(t,x)$. During the
proof we fix an $\omega \in \Omega _0$. Then, obviously, $W$ has a continuous path and its range over
$[0, K]$ lies in a ball $B_R(a) := \{x : |x-a| \le R\}$ with a finite radius $R = R(\omega)$. Also, by
Corollary \ref{co:strong_d}, we may assume that the range of $B_m$ over $[0, K]$ lies in the same ball
for any $m \ge m_0(\omega)$.

Consider \emph{the term on the left  side} of (\ref{eq:disc_Ito_planar}). We want to show that it
uniformly converges to $g(\omega , t, W(t)) - g(\omega , t, a)$ for $t \in [0,K]$. Define the path
$\gamma_n = [a, (B_m^1(t_n), a^2)]+[(B_m^1(t_n), a^2), B_m(t_n)]$, where $n:=\lfloor t 2^{2m} \rfloor$
and $t_n = n 2^{-2m}$. Then we have
\begin{eqnarray} \label{eq:trapez}
\lefteqn{\sup_{t\in [0,K]} \left|T_{x=a}^{B_m(t_{n})} \left(f_m(\omega , t_{n}, x), \gamma '_0(x)
\right) 2^{-m} - \int_{a}^{W(t)} \left((\nabla g)(\omega , t, x), \gamma '_0(x)
\right) \di s \right|} \nonumber \\
&\le& \sup_{t\in [0,K]} \left|T_{\gamma_n } \left(f_m(\omega , t_{n}, x), \gamma '_0(x) \right) 2^{-m}
- T_{\gamma_n } \left((\nabla g)(\omega , t, x), \gamma '_0(x) \right) 2^{-m}
\right| \nonumber \\
&+& \sup_{t\in [0,K]} \left|T_{\gamma_n} \left((\nabla g)(\omega , t, x), \gamma '_0(x) \right) 2^{-m}
- \int_{\gamma_n} \left((\nabla g)(\omega , t, x), \gamma '_0(x) \right) \di s \right|
\nonumber \\
&+& \sup_{t\in [0,K]} \left| \int_{B_m(t_{n})}^{W(t)} \left( (\nabla g)(\omega , t, x) , \gamma '_0(x)
\right) \di s \right|.
\end{eqnarray}

The first term on the right side of (\ref{eq:trapez}) can be bounded above by $K 2^{2m} \times 2^{-m}
\times (R/6) 2^{-m} \epsilon_K^0(2^{-m}) = O\left(\epsilon_K^0(2^{-m})\right)$. Here the first factor
bounds the number of terms in the trapezoidal sum, the second factor is a multiplier in each term, and
the third factor bounds the difference of the terms in the two sums by (\ref{eq:fm_g}).

The second term on the right side of (\ref{eq:trapez}) can be bounded by $K 2^{2m} \times (M_3/12)
2^{-3m} = O\left(2^{-m}\right)$. Here the first factor bounds the number of terms in the trapezoidal
sum and the second factor bounds the difference of a trapezoidal term and the corresponding integral
by (\ref{eq:trapez_error}), where $M_3$ is an upper bound of the magnitude of third $x$-partial
derivatives of $g$ for $(t,x) \in [0, K] \times B_R(a)$.

The third term on the right side of (\ref{eq:trapez}) can be bounded by $M_2 \times (O(m^{\frac34}
2^{-\frac{m}{2}})$ $+ 2^{-m} ) = O(m^{\frac34} 2^{-\frac{m}{2}})$. Here $M_2$ is an upper bound of
$|\nabla g|$ for $(t,x) \in [0, K] \times B_R(a)$, while the first term in the parentheses bounds
$|W(t)-B_m(t)|$ by (\ref{eq:strong_d}), and the second term bounds $|B_m(t)-B_m(t_n)|$.

In sum, the the term on the left  side of (\ref{eq:disc_Ito_planar}) is bounded by
$O\left(\epsilon_K^0(2^{-m})\right) + O(m^{\frac34} 2^{-\frac{m}{2}})$.

Let us turn now to \emph{the first term on the right  side} of (\ref{eq:disc_Ito_planar}). Define the
paths $\gamma_r = [a, (B_m^1(t_r), a^2)]+[(B_m^1(t_r), a^2), B_m(t_r)]$ $(1 \le r \le n)$. Then we
have
\begin{multline*}
\sum_{r=1}^{n}  T_{x=a}^{B_m(t_r)}  \left( \left\{f_m(\omega , t_r, x)
- f_m(\omega , t_{r-1}, x) \right\}, \gamma'_0(x) \right) 2^{-m}  \\
= \sum_{r=1}^{n}  T_{\gamma_r} \left( \left\{(D_t f_m^1(\omega , t_{r-1}+s_r^1, x),
D_t f_m^2(\omega , t_{r-1}+s_r^2, x) \right\} 2^{-2m}, \gamma'_0(x) \right) 2^{-m}  \\
= \sum_{r=1}^{n}  T_{\gamma_r} \left( \left\{(D_t D_1 g(\omega , t_{r-1}+s_r^1, x),
D_t D_2 g(\omega , t_{r-1}+s_r^2, x) \right\} , \gamma'_0(x) \right) 2^{-m} \: 2^{-2m} \\
+ O(\epsilon_K^1(2^{-m})) \\
= \sum_{r=1}^{n} D_t \left\{ T_{\gamma_r} \left(\nabla g(\omega , t_r, x),
\gamma'_0(x) \right) 2^{-m} \right\} \: 2^{-2m} + O(\epsilon_K^1(2^{-m})) + O(2^{-2m}) \\
= \sum_{r=1}^{n} D_t \int_{\gamma_r} \nabla g(\omega , t_r, x) \di s \: 2^{-2m}
+ O(\epsilon_K^1(2^{-m})) + O(2^{-m}) \\
= \sum_{r=1}^{n} D_t \left\{ g(\omega , t_r, B_m(t_r)) - g(\omega , t_{r}, a) \right\}
\: 2^{-2m} + O(\epsilon_K^1(2^{-m})) + O(2^{-m}) \\
= \int_0^t (D_t g)(\omega , u, W(u)) \di u - g(\omega , t, a) + g(\omega , 0, a) +
O(\epsilon_K^1(2^{-m})) + O(m^{\frac34} 2^{-\frac{m}{2}}) .
\end{multline*}
Above we made use of the fact that all considered functions are uniformly continuous over the bounded
set $[0, K] \times B_R(a)$. The first equality used the mean value theorem in the time variable,
component-wise for  $f_m$, with $0 \le s_r^1, s_r^2 \le 2^{-2m}$. The second equality applied
inequality (\ref{eq:field_modif_t}), combined with the largest possible number of terms in the sums
and the corresponding multipliers. The third equality estimated the error, when one replaces the
values of the components of $\nabla g$ at time $t_{r-1}+s_r^j$ by their values at time $t_r$. The
fourth equality replaced the trapezoidal sum by the corresponding integral, using
(\ref{eq:trapez_error}). In the fifth equality we evaluated the integral over the path $\gamma_r$. The
last equality replaces the sum of the function $D_t g$ in the time variable by an integral, and at the
same time replaces $B_m$ by $W$, using (\ref{eq:strong_d}).

Finally, let us consider \emph{the last term} in (\ref{eq:disc_Ito_planar}). By (\ref{eq:fm_g}), for
the first term in the braces we have
\begin{eqnarray*}
\lefteqn{\frac{f_m^1 \left(\omega , t_{r-1}, B_m^1(t_r), B_m^2(t_{r-1}) \right) - f_m^1 \left(\omega ,
t_{r-1},
B_m(t_{r-1})) \right)} {2^{-m} X^1_m(r)}} \\
&=& \frac{(D_1 g)\left(\omega , t_{r-1}, B_m^1(t_r), B_m^2(t_{r-1}) \right) - (D_1 g)\left(\omega ,
t_{r-1},
B_m(t_{r-1}) \right)} {2^{-m} X^1_m(r)} \\
&& + \: O(\epsilon_K^0(2^{-m})) \\
&=& (D_{11} g)\left(\omega , t_{r-1}, B_m^1(t_{r-1})+s^1_r , B_m^2(t_{r-1}) \right) +
O(\epsilon_K^0(2^{-m})),
\end{eqnarray*}
where $|s^1_r| \le 2^{-m}$. For the other part in the braces of the last term of
(\ref{eq:disc_Ito_planar}) involving $f^2_m$ one can obtain a similar result by the help of $D_{22}$.
In sum, with $|s^j_r| \le 2^{-m}$, we have
\begin{eqnarray} \label{eq:last_term}
\lefteqn{\sum_{r=1}^{n} \left\{ \frac{f_m^1 \left(\omega , t_{r-1}, B_m^1(t_r), B_m^2(t_{r-1}) \right)
- f_m^1 \left(\omega , t_{r-1}, B_m(t_{r-1}) \right)} {2^{-m} X^1_m(r)}
\right. } \nonumber \\
& & + \left. \frac{f_m^2 \left(\omega , t_{r-1}, B_m(t_r) \right) - f_m^2 \left(\omega , t_{r-1},
B_m^1(t_r), B_m^2(t_{r-1}) \right)} {2^{-m} X^2_m(r)}\right\} 2^{-2m}
\nonumber \\
&=& \sum_{r=1}^{n} \left\{(D_{11} g)\left(\omega , t_{r-1}, B_m^1(t_{r-1}) + s^1_r , B_m^2(t_{r-1})
\right) \right.
\nonumber \\
&& + \left. (D_{22} g)\left(\omega , t_{r-1}, B_m^1(t_{r}), B_m^2(t_{r-1}) + s^2_r  \right)
\right\} 2^{-2m} + O(\epsilon_K^0(2^{-m})) \nonumber \\
& = & \int_{0}^{t} (\Delta g)(\omega , u, W(u)) \di u + O(\epsilon_K^0(2^{-m})) + O(m^{\frac34}
2^{-\frac{m}{2}}).
\end{eqnarray}
Here, similarly as above, the sum was replaced by an integral and $B_m$ by $W$.

Thus we have seen that all terms, except for the stochastic sum, converge to their counterparts in
(\ref{eq:Ito_form}), almost surely uniformly on $[0, K]$. (An extra term $-g(\omega , t, a)$ has
appeared too on both sides, that cancel each other.) Therefore the stochastic sum must converge to the
stochastic integral in the same sense as well. This ends the proof of the theorem.

\end{proof}

The reader may have noticed in the statement of Theorem \ref{th:planar_Ito} that the usual condition
in It\^o's formulae that the random function $g(\omega , t, x)$ be adapted to the filtration of
Brownian motion $W$, was not needed: the assumed smoothness of $g$ together with the pathwise,
integration by parts stochastic integration technique made this assumption unnecessary.

%%%%%%%%%%%%%%%%%%%%%%%%%%%%%%%%%%%%%%%%%%%%%%%%%%%%%%%%%%%%%%%%%%%

\section{Discrete self-intersection local time}
\label{sec:DSILT}

The definition of discrete \emph{self-intersection local time} follows the lines of the definition of
discrete local time (\ref{eq:dloct}), (\ref{eq:tasloct}), (\ref{eq:DPLT}) and (\ref{eq:DLT}). Take
first a planar simple, symmetric random walk $(S_n)_{n=0}^{\infty}$, $S_0=0$, $S_n = \sum_{i=1}^{n}
X_i$ ($n \ge 1)$, where $(S_{n}^{1})$ and $(S_{n}^{2})$ are independent one-dimensional simple,
symmetric random walks with unit steps in unit time. Define \emph{self-intersection local time of the
random walk} by
\begin{multline}\label{eq:SILT}
\alpha_1(n, x) := \#\{(i,j) : 0 \le i \le j < n, \: S_j - S_i = x \} \\
= \sum_{i=0}^{n-1} \sum_{j=i}^{n-1} \mathbf{1}_{\{S_j-S_i=x\}},
\end{multline}
where $n \in \mathbb{Z}_+$ and $x \in \mathbb{Z}^2$. We also need \emph{partial self-intersection
local times}
\begin{equation}\label{eq:PSILT}
\alpha_1^{\mu}(n, x) := \#\{(i,j) : 0 \le i \le j < n, \: S_j - S_i = x, \: S_{j+1} - S_j = \mu  \},
\end{equation}
where $\mu \in \{-1, 1\}^2$.

Clearly, by the strong Markovian property of random walks, each inner sum $\ell_i(n-i,x):=
\sum_{j=i}^{n-1} \mathbf{1}_{\{S_j-S_i=x\}}$ in the last term of (\ref{eq:SILT}) is a local time of a
random walk started from the point $S_i$, taken at time $n-i$ at the point $x$;
\begin{equation}\label{eq:sumloct}
\alpha_1(n, x) = \sum_{i=0}^{n-1} \ell_i(n-i,x) .
\end{equation}
Denote the largest number of visits to a point of the random walk in the first $n$ steps by
\[
\ell^*(n) := \sup_{x \in \mathbb{Z}^2} \ell(n,x).
\]
Similarly, denote the largest number of visits to a point of the random walk starting from point
$S_i$, in the first $n-i$ steps, by $\ell_i^*(n-i)$. Then for any $\omega \in \Omega$ one clearly has
\begin{equation}\label{eq:loctseq}
\ell_0^*(n) \ge \ell_{1}^{*}(n-1) \ge \cdots \ge \ell_{n-1}^{*}(1) .
\end{equation}

In a classical paper \cite{ET1960}, Erd\H{o}s and Taylor showed the following inequality for the
maximum number of visits of a random walk in the first $n$ steps:
\begin{equation}\label{eq:maxloct}
\limsup_{n \to \infty} \frac{\ell^{*}(n)}{\log^2(n)} \le \frac{1}{\pi} \qquad \text{a.s.}
\end{equation}
The next lemma is an easy consequence of this result.
\begin{lem} \label{le:SILT}
\[
\limsup_{n \to \infty} \frac{\sup_{x \in \mathbb{Z}^2} \alpha_1(n, x)}{n \log^2 n}  \le \frac{1}{\pi}
\qquad \text{a.s.}
\]
\end{lem}
\begin{proof}
By (\ref{eq:sumloct}), (\ref{eq:loctseq}) and (\ref{eq:maxloct}),
\[
\limsup_{n \to \infty} \frac{\sup_{x \in \mathbb{Z}^2} \alpha_1(n, x)}{n \log^2 n}  \le \limsup_{n \to
\infty} \frac{n \ell_0^*(n)}{n \log^2 n} \le \frac{1}{\pi}  \qquad \text{a.s.}
\]
\end{proof}
I do not know if this lemma is sharp or what sharp lower limit rate could be given for $\sup_{x \in
\mathbb{Z}^2} \alpha_1(n, x)$. In this regard it can be mentioned that Dembo et al. \cite{DPRZ2001}
relatively recently proved the conjecture of Erd\H{o}s and Taylor that in fact
\[
\lim_{n \to \infty} \frac{\ell^{*}(n)}{\log^2(n)} = \frac{1}{\pi}  \qquad \text{a.s.}
\]

Now we apply the previous results to shrunken random walks. Let $h > 0$, $x \in h \mathbb{Z}^2 =
\mathcal{G}(0,h)$ and $t \in h^2 \mathbb{Z}_+$. Consider a simple, symmetric random walk
$(S_n)_{n=0}^{\infty}$, $S_0=0$, on the grid $\mathcal{G}(0,h)$, with time steps $h^2$. (That is, the
time between step $n$ and step $n+1$ of the walk is $h^2$.) Define the corresponding
\emph{self-intersection local time} as
\begin{multline} \label{eq:TDSILT}
\alpha _{h} (t, x) := h^2 \: \#\{(i,j) : 0 \le i \le j < t / h^2, \: S_j - S_i = x \} \\
= h^2 \sum_{i=0}^{n-1} \sum_{j=i}^{n-1} \mathbf{1}_{\{S_j - S_i = x\}} = h^2 \sum_{i=0}^{n-1} \ell_i(n
- i, x/h),
\end{multline}
where $n = t / h^2$ and $\ell_i(n - i, x/h)$ is defined in the same way as above. A \emph{partial
self-intersection local time} in the direction $\mu \in \{1, -1 \}^2$ is
\begin{equation} \label{eq:PDSILT}
\alpha _{h}^{\mu}(t, x) := h^2 \sum_{i=0}^{n-1} \sum_{j=i}^{n-1} \mathbf{1}_{\{S_j - S_i = x; \:
S_{j+1} - S_j = \mu h \}} .
\end{equation}

Let us extend $\alpha _{h} (t, x)$ for any $t \in \mathbb{R}_+$ and $x \in \mathbb{R}^2$ as a
continuous function.  First, with $t \in h^2 \mathbb{Z}_+$ fixed, we apply linear interpolation in
$x$. Let $x$ be a point in a lower triangle $\Delta $ with vertices $(a^1,a^2)$, $(a^1+h,a^2)$, and
$(a^1,a^2+h)$ for some $a \in h \mathbb{Z}^2$. Let $A=\alpha _{h}(t, (a^1,a^2))$, $B=\alpha _{h}(t,
(a^1+h,a^2))$, and $C=\alpha _{h}(t, (a^1,a^2+h))$. Then define
\begin{equation}\label{eq:int_pol}
\alpha _{h}(t, x) := A + \frac{x^1 - a^1}{h}(B-A) + \frac{x^2 - a^2}{h}(C-A) .
\end{equation}
Analogous is the case with an upper triangle.

Second, define $\alpha _{h} (t, x) := \alpha _{h} (h \lfloor t/h \rfloor, x)$ for $t \in \mathbb{R}_+$
and $x \in \mathbb{R}^2$. Similar is the extension of partial self-intersection local times as
continuous functions. It will be of use later that then
\begin{equation}\label{eq:int_pol_int}
\int_{\Delta } \alpha _{h}(t, x) \di x = \frac{h^2}{6} (A+B+C).
\end{equation}

Lemma \ref{le:SILT} clearly implies that
\[
\limsup_{h \to 0^+} \frac{\sup_{x \in \mathbb{R}^2} \alpha_h(t, x)}{t \log^2 (t/h^2)}  \le
\frac{1}{\pi} \qquad \text{a.s.}
\]

Briefly, this means that
\begin{equation}\label{eq:supalpha}
\sup_{x \in \mathbb{R}^2} \alpha_h(t, x) = O\left( \log^2 (h) \right) \qquad (h \to 0^+) \qquad
\text{a.s.}
\end{equation}

Take now a planar Brownian motion $(W(t))_{t \ge 0}$, $W(0)=0$. Then take planar Skorohod embedded
random walks $(B_m(t))_{t \ge 0}$, $B_m^j(r2^{-2m})=W^j(s_m^j(r))$ ($j=1,2$) for $m \in \mathbb{Z}_+$.
Clearly, $B_m(t)$ is a shrunken random walk with $h = 2^{-m}$. For sake of simplicity, let us denote
the corresponding self-intersection local time by $\alpha _{m} (t, x)$, and partial  self-intersection
local times by $\alpha _{m}^{\mu}(t, x)$. It follows that
\begin{equation}\label{eq:ordSILT}
\sup_{x \in \mathbb{R}^2} \alpha_m(t, x) = O(m^2) \qquad  (m \to \infty) \qquad \text{a.s.},
\end{equation}
uniformly on any bounded time interval $t \in [0, K]$.

It will be also useful in the sequel that for a.e. $\omega $, $\alpha _{m}(t, x) = 0$ if $|x| > R =
R(\omega )$, for any $m \ge 0$ and $t \in [0, K]$, supposing $R$ is large enough. This follows from
the fact that for a.e. $\omega \in \Omega$, the continuous function $W(v)-W(u)$ is bounded on the
compact triangle $V_K = \{(u,v) : 0 \le u \le v \le K \}$ and $B_m(v) - B_m(u)$ almost surely
uniformly converges to it on $V_K$ as $m\to \infty$.

%%%%%%%%%%%%%%%%%%%%%%%%%%%%%%%%%%%%%%%%%%%%%%%%%%%%%%%%%%%%%%%%%%%

\section{A discrete Tanaka--Rosen--Yor formula}
\label{sec:dTRY}

The aim of this section is to give a discrete version of the planar Tanaka--Rosen--Yor formula. Beyond
its intrinsic interest, a special case of this formula will serve as a basic tool for a rather natural
definition of planar Brownian self-intersection local time. Like the discrete It\^o's formulae in
Section \ref{sec:disc_Ito}, this formula will be an algebraic--analytic one with no intrinsic
randomness involved, though, naturally, it will be applied in a probabilistic context afterward.

\begin{lem} \label{le:disc_TRY}
Let $\phi$ be a $C^3$ scalar field in the plane. Fix an $h > 0$ and let $x \in h \mathbb{Z}^2 =
\mathcal{G}(0,h)$. Consider a sequence $X_r = (X^1_{r}, X^2_{r})$ ($r \ge 1$), where $X^j_{r} = \pm
1$. Take partial sums $S_0=0 $, $S_n = h(X_1+\cdots +X_n)$ ($n \ge 1$), supposing that the steps of
this ``walk'' are performed in time units $h^2$. Let us take the discrete paths $\gamma_r = [0,
(S_r^1,0)] + [(S_r^1,0), S_r]$, ($1 \le r \le n$). Then with any $y \in h \mathbb{Z}^2$ fixed, one
obtains the following \emph{discrete Tanaka--Rosen--Yor formula}:
\begin{multline} \label{eq:dTRY1}
\sum_{j=0}^{n} \left\{ T_{\gamma _n} \left( \nabla \phi (x - S_j - y) , \gamma '_0(x) \right)
h \right\} \: h^2 \\
= \sum_{r=1}^{n} \left\{ T_{\gamma _r} \left( \nabla \phi (x - S_{r} - y) , \gamma '_0(x) \right)
h \right\} \: h^2 \\
+ \sum_{r=1}^{n} \sum_{j=0}^{r-1} \left\{ (D_1 \phi)\left(S_{r-1} - S_j - y \right) h X_r^1
+ (D_2 \phi)\left((S_{r}^{1}, S_{r-1}^2) - S_j - y \right) h X_r^2 \right\} h^2 \\
+ \frac12 \sum_{r=1}^n \sum_{j=0}^{r-1} (\Delta \phi)\left(S_{r-1} - S_j - y\right) \: h^4 + O(h) +
O(\epsilon (h)),
\end{multline}
where $\epsilon (h) \to 0$ as $h \to 0$. One has the following equality for the last term as well:
\begin{multline} \label{eq:dTRY2}
L_h \phi (t_n, y) := \sum_{r=1}^n \sum_{j=0}^{r-1} (\Delta \phi)\left(S_{r-1} - S_j - y\right) \: h^4 \\
= \sum_{x \in h \mathbb{Z}^2}  \alpha _{h}(t_n, x) \: (\Delta \phi)(x - y) \: h^2 + O(h \log^2 h) ,
\end{multline}
where $t_n = n h^2$ and $\alpha _{h}(t_n, x)$ is the self-intersection local time (\ref{eq:SILT}) of
the sums $S_n$. For any $K>0$ fixed, the error terms in (\ref{eq:dTRY1}) and (\ref{eq:dTRY2}) are
uniform while $t_n \in [0, K]$.
\end{lem}

\begin{proof}
Define the following time dependent scalar field $g^y : h^2 \mathbb{Z}_+ \times h \mathbb{Z}^2 \to
\mathbb{R}$,
\[
g^y(t,x) := \sum_{j=0}^{t/h^2} \phi(x - S_j - y) \: h^2,
\]
where $y \in h \mathbb{Z}^2$ is a parameter. Take a finite $R > 0$ such that the disc $B_R(0)$ cover
all the points $(S_r)_r=0^n$, $n=t/h^2$ and the point $y$ as well. Then all points $S_j - S_i -y$ are
contained by the disc $B_{3R}(0)$. Thus by Lemma \ref{le:error_in_field}, one can construct a discrete
conservative vector field $\psi $ in the plane such that
\begin{equation} \label{eq:grad_error}
\sup_{|x| \le 3R} \left| \psi (x) - \nabla \phi(x) \right| \le \frac{R}{2} h \epsilon (h) ,
\end{equation}
where $\epsilon (h) = \epsilon _{\phi}(h,R) \to 0$ as $h \to 0$.

Further, define
\[
f^y(t,x) := \sum_{j=0}^{t/h^2} \psi (x - S_j - y) \: h^2 .
\]
Then by (\ref{eq:grad_error}) it follows that for any $y$ and $t$ fixed,
\begin{multline} \label{eq:sumgrad_error}
\sup_{|x| \le 3R} \left| f^y(t,x) - \nabla g^y(t,x) \right| \\
= \sup_{|x| \le 3R} \left| \sum_{j=0}^{t/h^2} \psi (x - S_j - y) \: h^2 - \sum_{j=0}^{t/h^2} \nabla
\phi(x - S_j - y) \: h^2 \right| \le \frac{Rt}{2} h \epsilon (h) ,
\end{multline}
and $f^y(t,x)$ is a discrete conservative vector field in the plane.

Now apply the discrete It\^o's formula (\ref{eq:disc_Ito}) to $f^y(t,x)$. Let us denote $t_r = r h^2$
$(r = 0,1, \dots, n)$. Then we get the following terms.

\emph{The term on the left side} of (\ref{eq:disc_Ito}) becomes
\begin{multline*}
T_{x=S_0}^{S_n} \left(f^y(t_n,x), \gamma '_0(x) \right) h \\
= \sum_{j=0}^{n} \left\{ T_{\gamma _n} \left( \nabla \phi (x - S_j - y) , \gamma '_0(x) \right) h
\right\} \: h^2 + O(\epsilon (h)) .
\end{multline*}
The error term is obtained since there are $t/h^2$ terms in the trapezoidal sum, there is a multiplier
$h$ in each term, and we can apply (\ref{eq:sumgrad_error}) to each term.

\emph{The first term on the right side} of (\ref{eq:disc_Ito}) becomes
\begin{multline*}
\sum_{r=1}^{n}  T_{x=S_0}^{S_r} \left(\left\{f^y(t_r,x) - f^y(t_{r-1},x)\right\}, \gamma'_0(x) \right) h \\
= \sum_{r=1}^{n} \left\{ T_{\gamma _r} \left( \nabla \phi (x - S_{r} - y) , \gamma '_0(x) \right) h
\right\} \: h^2 + O(\epsilon (h)) .
\end{multline*}
The error term is obtained since there are $t/h^2$ terms in both sums, respectively; there are
multipliers $h$ and $h^2$ in each term, respectively; and we can apply (\ref{eq:grad_error}) for each
term.

\emph{The second term on the right side} of (\ref{eq:disc_Ito}) becomes
\begin{multline*}
\sum_{r=1}^{n} \left\{ (f^y)^1 \left(t_{r-1}, S_{r-1} \right) h X_r^1 + (f^y)^2 \left(t_{r-1},
(S_{r}^{1}, S_{r-1}^2) \right) h X_r^2 \right\} \\
= \sum_{r=1}^{n} \sum_{j=0}^{r-1} \left\{ (D_1 \phi)\left(S_{r-1} - S_j - y \right) h X_r^1
+ (D_2 \phi)\left((S_{r}^{1}, S_{r-1}^2) - S_j - y \right) h X_r^2 \right\} h^2 \\
+ O(\epsilon (h)) .
\end{multline*}
Again, the error term is obtained since there are at most $t/h^2$ terms in both sums, respectively;
there are multipliers $h$ and $h^2$ in each term, respectively; and we can apply (\ref{eq:grad_error})
for each term.

Finally, \emph{the last term on the right side} of (\ref{eq:disc_Ito}) of (\ref{eq:disc_Ito})becomes
\begin{multline} \label{eq:last_term0}
\sum_{r=1}^n \left\{ \frac{(f^y)^1 \left(t_{r-1}, (S_r^1, S_{r-1}^{2}) \right) - (f^y)^1
\left(t_{r-1}, S_{r-1} \right)}{hX_r^1} \right. \\
\left. + \frac{(f^y)^2 \left(t_{r-1}, S_{r} \right) - (f^y)^2 \left(t_{r-1},(S_r^1, S_{r-1}^{2})
\right)}{hX_r^2} \right\} h^2 \\
= \sum_{r=1}^n \sum_{j=0}^{r-1} \left\{ \frac{(D_1\phi)\left((S_r^1, S_{r-1}^{2}) - S_j - y \right)
- (D_1\phi) \left(S_{r-1} - S_j - y \right)}{hX_r^1} \right. \\
\left. + \frac{(D_2\phi) \left(S_{r} - S_j - y \right) - (D_2\phi) \left((S_r^1, S_{r-1}^{2}) - S_j -
y \right)}{hX_r^2} \right\} h^4 + O(\epsilon (h)) .
\end{multline}
Here the error term is obtained because there are at most $t/h^2$ terms in both sums, respectively;
there are multipliers $h^2$ in both, respectively; each term is divided by $h$; and we can apply
(\ref{eq:grad_error}) for each term.

We need to write the last term in two different ways. The first way mimics the method applied to the
last term in the proof of Theorem \ref{th:planar_Ito}. There exist $s_{r}^{1}, s_{r}^{2} \in [-h, h]$
such that the last term equals
\begin{multline} \label{eq:last_term1}
\sum_{r=1}^n \sum_{j=0}^{r-1} \left\{ (D_{11}\phi)\left((S_{r-1}^1 + s_{r}^{1},
S_{r-1}^{2}) - S_j - y)\right) \right. \\
\left. + (D_{22}\phi) \left((S_r^1, S_{r-1}^{2} + s_{r}^{2}) - S_j - y \right)\right\} h^4
+ O(\epsilon (h)) \\
= \sum_{r=1}^n \sum_{j=0}^{r-1} (\Delta \phi)\left(S_{r-1} - S_j - y\right) \: h^4 + O(h) +
O(\epsilon(h)).
\end{multline}
Here the error term $O(h)$ is obtained when one replaces the translations $s_r^j$ by 0 in the second
partial derivatives of $\phi$, which are uniformly continuous over the bounded ball $B_{3R}(0)$.

The second way of writing the last term (\ref{eq:last_term0}) uses discrete self-intersection local
times (\ref{eq:PDSILT}), arranging the terms in (\ref{eq:last_term0}) according to $x = S_{r-1} -
S_j$:
\begin{multline} \label{eq:last_term2}
\sum_{x \in h \mathbb{Z}^2} \sum_{\mu \in \{-1, 1\}^2}  \alpha _{h}^{\mu}(t_n, x) \: u^{\mu}(x - y)
\: h^2 + O(\epsilon(h)) \\
= \sum_{x \in h \mathbb{Z}^2}  \alpha _{h}(t_n, x) \: (\Delta \phi)(x - y) \: h^2 + O(h \log^2 h) +
O(\epsilon(h)),
\end{multline}
where
\begin{multline} \label{eq:umux}
u^{\mu}(x) := \frac{(D_1\phi)(x^1+h\mu^1,x^2) - (D_1\phi)(x^1,x^2)}{h\mu^1} \\
+ \frac{(D_2\phi)(x^1+h\mu^1,x^2+h\mu^2) - (D_2\phi)(x^1+h\mu^1,x^2)}{h\mu^2} \\
= (D_{11}\phi)(x^1+s^1,x^2) + (D_{22}\phi)(x^1+h\mu^1, x^2+s^2) = (\Delta \phi)(x) + O(h),
\end{multline}
$s^1, s^2 \in [-h, h]$. In (\ref{eq:last_term2}) the error term $O(h \log^2 h)$ is obtained when one
replaces the translations $s_r^j$ by 0 in the second partial derivatives of $\phi$ in (\ref{eq:umux}),
harnessing the upper bound (\ref{eq:supalpha}) for $\alpha_{h}(t_n, x)$ and the fact that there are at
most $t_n/h^2$ non-zero terms (multiplied by $h^2$) in the summation for $x$.

Then, collecting the terms of the discrete It\^o's formula, by (\ref{eq:last_term1}) and
(\ref{eq:last_term2}) we obtain the two versions claimed in the lemma.
\end{proof}

It has to be emphasized that so far in this section all obtained formulae have been
algebraic--analytic ones, independent of any randomness. Now take a planar Brownian motion $W(t)$ and
replace the sums above by a Skorohod imbedded sequence: $S_n = B_m(n 2^{-2m}))$, with $h=2^{-m}$. Then
taking limits of the discrete Tanaka--Rosen--Yor formula as $m \to \infty$, one obtains a continuous
version of the formula \cite[Th\'eor\`eme 1]{Yor1985}.

\begin{thm}\label{th:TRY}
Suppose that $\phi$ is a $C^3$ scalar field, $W(t)$ is a Brownian motion in the plane, $W(0)=0$, and
$y \in \mathbb{R}^2$. With $m =0,1,\dots$ and $h=2^{-m}$, apply the discrete Tanaka--Rosen--Yor
formula (\ref{eq:dTRY1}) to the imbedded random walks $S_n = B_m(n 2^{-2m}))$. Then, as $m \to
\infty$, each term of the discrete formula almost surely tends to the corresponding term of the
following continuous formula, uniformly on any bounded interval $t \in [0, K]$:
\begin{multline}\label{eq:TRY}
\int_{0}^{t} \phi\left(W(t) - W(u) - y\right) \di u \\
= t \phi(-y) + \int_{0}^{t} \int_{0}^{v} (\nabla \phi)\left(W(v) - W(u) - y\right) \di u \cdot \di
W(v)
\\
+ \frac12 \int_{0}^{t} \int_{0}^{v} (\Delta \phi)\left(W(v) - W(u) - y\right) \di u \di v .
\end{multline}
The last term can be written as an almost sure limit of sums involving self-intersection local times
of imbedded random walks:
\begin{multline}\label{eq:TRYlast}
L\phi(t,y) := \int_{0}^{t} \int_{0}^{v} (\Delta \phi)\left(W(v) - W(u) - y\right) \di u \di v \\
= \lim_{m \to \infty} \sum_{x \in 2^{-m}\mathbb{Z}^2} \alpha _m(t, x) \: (\Delta \phi)(x-y) \: 2^{-2m}
.
\end{multline}
\end{thm}

\begin{proof}
To prove the almost sure convergence of the terms in (\ref{eq:TRY}), it is essentially enough to apply
Theorem \ref{th:planar_Ito}. The only new element here is that in each term there is a Riemann sum of
a continuous function, that converges to the corresponding Riemann integral for almost every path as
$m \to \infty$. The value of the parameter $y \in \mathbb{R}^2$ should also be approximated by a
closest point $y_m \in 2^{-m}\mathbb{Z}^2$, for which $|y - y_m| \le 2^{-m}$. This does not cause any
problem, since all functions in (\ref{eq:TRY}) are continuous in $y$.

Thus the limit of the term on the left side of (\ref{eq:dTRY1}) is
\[
\int_{0}^{t} \left\{ \phi\left(W(t) - W(u) - y\right) - \phi\left(- W(u) - y\right)\right\} \di u .
\]

The limit of the first term on the right side is
\[
\int_{0}^{t} \left\{ \phi\left(W(u) - W(u) - y\right) - \phi\left(- W(u) - y\right)\right\} \di u.
\]
The extra $\int_{0}^{t} - \phi\left(- W(u) - y\right) \di u$ term appears on both sides, so can be
canceled. Observe that $\int_{0}^{t} \phi\left(W(u) - W(u) - y\right) = t \phi(-y)$.

The limits of the second and the last terms of the right side of (\ref{eq:dTRY1}) are clearly the
corresponding ones in (\ref{eq:TRY}). The equality in (\ref{eq:TRYlast}) clearly follows from the
equality (\ref{eq:dTRY2}).
\end{proof}

Since $\phi \in C^3(\mathbb{R}^2)$ and $W(t)$ is a.s. continuous, it follows that, almost surely, the
term on the left side and the first and the last terms on the right side of (\ref{eq:TRY}) are
continuous functions of $(t, y) \in \mathbb{R_+} \times \mathbb{R}^2$. This implies the same
conclusion for the second, stochastic integral term as well.

%%%%%%%%%%%%%%%%%%%%%%%%%%%%%%%%%%%%%%%%%%%%%%%%%%%%%%%%%%%%%%%%%%%

\section{A definition of planar self-intersection local time}
\label{sec:def_SILT}

A possible definition of ordinary local time in one spatial dimension uses a special case of Tanaka's
formula applied with the function $\phi(x) = x \vee 0$, which is a fundamental solution of the
one-dimensional Laplacian $ d^2/dx^2$; see this kind of definition for example in \cite[p.
117]{SV1979} and \cite[Section 7.2]{CW1990}. The definition of self-intersection local time presented
below is a suitable planar modification of it. This means that our definition uses a special case of
planar Tanaka--Rosen--Yor formula with the function $\phi(x) = \log |x|$, which is a fundamental
solution of the planar Laplacian $\Delta = D_{11} + D_{22}$, ignoring a constant multiplier.

Let $W(t)$ be a planar Brownian motion, $W(0)=0$.  For each $x \in \mathbb{R}^2$ and $\delta > 0$, we
define an everywhere continuously differentiable approximation of $\log|x|$ by
\[
\phi^{\delta }(x) := \left\{\begin{array}{ll}
\frac{|x|^2 - \delta ^2}{2 \delta ^2} + \log \delta & \text{for } |x| \le \delta,  \\
\log|x| & \text{for } |x| \ge \delta .
\end{array} \right.
\]
Then
\begin{equation}\label{eq:nablaphi}
(\nabla \phi^{\delta })(x) = \left\{\begin{array}{ll}
\frac{x}{\delta ^2} & \text{for } |x| \le \delta,  \\
\frac{x}{|x|^2} & \text{for } |x| \ge \delta ;
\end{array} \right.
\end{equation}
and
\[
(\Delta \phi^{\delta })(x) = \left\{\begin{array}{ll}
\frac{2}{\delta ^2} & \text{for } |x| < \delta,  \\
0 & \text{for } |x| > \delta .
\end{array} \right.
\]
Note that $(\Delta \phi^{\delta })(x)$ is not defined for $|x| = \delta $, but we set it to be $0$
there.

Since $\phi^{\delta }$ is not $C^3$, Theorem \ref{th:TRY} is not directly applicable to it. However,
by a standard procedure, taking a convolution with a sequence of $C^{\infty}$ functions $q_n$ with
compact support shrinking to $\{0\}$, and then taking a limit as $n \to \infty$, solves this problem.
For sake of explicitness, let $q(z) = c \exp\left(-(1-|z|^2)^{-1}\right)$ for $|z| < 1$ and $0$
otherwise, where the constant $c$ is chosen so that $\int_{\mathbb{R}^2} q(z) \di z =1$. Put $q_n(z) =
n^2 q(nz)$ and $\phi_{n}^{\delta } = \phi ^{\delta } \ast q_n$ $(n \ge 1)$.

Then $\phi_{n}^{\delta } \in C^{\infty}(\mathbb{R}^2)$; $\phi_{n}^{\delta } \to \phi^{\delta }$,
$\nabla \phi_{n}^{\delta } \to \nabla \phi^{\delta }$ both uniformly in $\mathbb{R}^2$; while $\Delta
\phi_{n}^{\delta } \to \Delta \phi^{\delta }$ pointwise except for $|x| = \delta $. Thus one can apply
Theorem \ref{th:TRY} to $\phi_{n}^{\delta }$, and take a limit of the terms as $n \to \infty$. The
resulting formula is
\begin{multline} \label{eq:spec_TRY}
\int_{0}^{t} \phi^{\delta }\left(W(t) - W(u) - y \right) \di u \\
= t \phi^{\delta }(y)
+ \int_{0}^{t} \int_{0}^{v} (\nabla \phi^{\delta })\left(W(v)-W(u)-y\right) \di u \cdot \di W(v) \\
+ \frac{1}{\delta ^2} \; \lambda \left( \{(u,v) \in V_t : |W(v)-W(u)-y| < \delta \} \right) ,
\end{multline}
where $\lambda $ denotes planar Lebesgue measure, $V_t = \{(u,v) : 0 \le u \le v \le t \}$, and $y \in
\mathbb{R}^2$. It also follows that each term here is an almost surely continuous function of $(t,y)$:
this is clear for each term except for the second, stochastic integral term on the right side, but
then it follows for this term too.

It is important that, by (\ref{eq:TRYlast}), the last term can be written as
\begin{equation}\label{eq:speclast}
\frac12 L\phi^{\delta}(t,y) = \frac{1}{\delta ^2} \; \lim_{m \to \infty} \sum_{x \in
2^{-m}\mathbb{Z}^2 \cap B_{\delta} (y)} \alpha _m(t, x) \: 2^{-2m},
\end{equation}
where $B_{\delta} (y)$ is the closed disc centered at $y$ with radius $\delta $.

Rosen \cite{Ros1983} suggested the following definition of \emph{self-intersection local time of
planar Brownian motion} $W$. Define the following \emph{occupation measure} for plane Borel sets $A$
and time $t\ge 0$:
\[
\mu_t(A) := \lambda\left(\{(u,v): 0\le u \le v \le t, W(v)-W(u) \in A \} \right) ,
\]
where $\lambda$ is planar Lebesgue measure. Rosen \cite{Ros1983} then proved that the
self-intersection local time $\alpha(t,x) := \frac{\di \mu_t}{\di \lambda}(x)$ a.s. exists when $x \ne
0$.

An alternative approach is to consider the \emph{symmetric derivative} of $\mu_t$ w.r.t. $\lambda$:
\begin{multline}\label{eq:defSILT}
\alpha (t, y) := \lim_{\delta \to 0^+} \frac{\mu_t(B_{\delta}(y))}{\lambda(B_{\delta}(y))} \\
= \lim_{\delta \to 0^+} \frac{1}{\pi \delta ^2} \; \lambda \left( \{(u,v) : 0 \le u \le v \le t,
|W(v)-W(u)-y| < \delta \} \right)
\end{multline}
where $B_{\delta}(y)$ denotes the disc centered at $y$ with radius $\delta$. Among other things, the
next theorem establishes the a.s. existence of this finite symmetric derivative for any $t \ge 0$ and
$y \ne 0$. It is well-known (see e.g. Rudin \cite{Rud1987}) that when there exists a finite symmetric
derivative of $\mu_t$ w.r.t. $\lambda$ except for the point $0$, then $\mu_t$ is absolutely continuous
w.r.t. $\lambda$ on $\mathbb{R}^2\setminus\{0\}$ and the symmetric derivative equals the
Radon--Nikodym derivative; the support of the singular part of $\mu_t$ can only be the point $0$.

\begin{thm} \label{th:spec_TRY}
The terms of (\ref{eq:spec_TRY}) almost surely converge to the corresponding terms of the following
Tanaka--Rosen--Yor formula as $\delta \to 0^+$, when $y \ne 0$:
\begin{eqnarray} \label{eq:TRYmain}
\lefteqn{\int_{0}^{t} \log |W(t) - W(u) - y| \di u} \nonumber \\
&=& t \; \log |y| + \int_{0}^{t} \int_{0}^{v} \frac{W(v)-W(u)-y}{|W(v)-W(u)-y|^2} \di u \cdot \di W(v)
+ \pi  \; \alpha(t, y) ,
\end{eqnarray}
cf. \cite[(2.j)]{Yor1985}. Moreover, all terms, including $\alpha(t, y)$, are a.s. continuous in
$(t,y)$ when $y \ne 0$.

It also follows that the self-intersection local time $\alpha(t, y)$ of planar Brownian motion is the
almost sure limit of averages of self-intersection local times $\alpha_m(t, y)$ of imbedded random
walks:
\begin{multline}\label{eq:SILTlim}
\alpha(t, y) = \lim_{\delta \to 0^+}  \lim_{m \to \infty } \frac{1}{\pi \delta^2} \sum_{ x \in 2^{-m}
\mathbb{Z}^2 \cap B_{\delta} (y)} \alpha _m(t, x) \: 2^{-2m} \\
= \lim_{\delta \to 0^+}  \lim_{m \to \infty } \frac{1}{\pi \delta^2} \int_{B_{\delta }(y)} \alpha_m(t,
x) \di x \qquad (y \ne 0).
\end{multline}
\end{thm}

\begin{proof}
The well-known properties of planar Brownian motion imply that for any $t > 0$ and $y \ne 0$, $\inf
\{|W(t) - W(u) - y| : 0 \le u \le t \} >  0$, with probability 1. Hence it follows the almost sure
convergence of the left side of (\ref{eq:spec_TRY}) as $\delta \to 0^+$ when $y \ne 0$. Moreover, the
integrand converges monotonically as $\delta \to 0^+$, so the limit and the integral can be
interchanged. It also follows that the left side of (\ref{eq:TRYmain}) is a continuous function of
$(t,y)$ when $y \ne 0$.

The convergence and the continuity of the first term on the right side is trivial when $y \ne 0$. .

By Lemma \ref{le:gauss} in the Appendix, the second term on the right side of (\ref{eq:spec_TRY}) is
\begin{multline} \label{eq:secondterm}
\int_{0}^{t} \int_{0}^{v} (\nabla \phi^{\delta })\left(W(v)-W(u)-y\right) \di u \cdot \di W(v)  \\
= \frac{1}{\pi \delta ^2} \int_{0}^{t} \int_{0}^{v} \int_{B_{\delta }(y)} \frac{W(v)-W(u)-z}
{|W(v)-W(u)-z|^2}
\: \di z \di u \cdot \di W(v) \\
= \frac{1}{\pi \delta ^2} \int_{B_{\delta }(y)} \left\{ \int_{0}^{t} \int_{0}^{v}
\frac{W(v)-W(u)-z}{|W(v)-W(u)-z|^2} \:  \di u \cdot \di W(v) \right\} \di z .
\end{multline}
The interchange of integrations is allowed by an extension of Fubini theorem. (Remember that by Lemma
\ref{le:disc_TRY}, the stochastic integral is an almost sure limit of discrete sums.) By Lemma
\ref{le:contversion} in the Appendix the stochastic integral has a continuous version for $y \ne 0$.
Thus by the mean value theorem of integrals, (\ref{eq:secondterm}) has an almost sure limit if $y \ne
0$ as $\delta \to 0^+$, namely the one stated in the theorem. By Lemma \ref{le:contversion}(c), the
limit is a continuous function of $(t,y)$ when $y \ne 0$.

The above limits imply that the last term on the right side of (\ref{eq:spec_TRY}) has an almost sure
limit $\pi \alpha (t, y)$ if $y \ne 0$. (\ref{eq:SILTlim}) follows from this by (\ref{eq:speclast})
and Lemma \ref{le:conv} in the Appendix. It also follows from the above arguments that $\alpha(t, y)$
is continuous in $(t,y)$ when $y \ne 0$.

\end{proof}

Formula (\ref{eq:SILTlim}) is the definition of planar self-intersection local time which has been the
main objective of the present paper. It is an open question if the limits in (\ref{eq:SILTlim}) can be
interchanged; then one would get the more impressive almost sure limit $\alpha(t, y) =  \lim_{m \to
\infty } \alpha _m(t, y)$.

The most interesting question is `What happens to the planar self-intersection local time
$\alpha(t,y)$ when $y \to 0$?' It was discovered by Varadhan in 1968 that $\alpha$ goes to $\infty$
then, and, in fact, it has a logarithmic singularity at 0. So one can introduce \emph{renormalized
self-intersection local time} $\gamma$ by the formula
\[
\gamma (t, y) = \left\{
\begin{array}{lll}
\alpha(t, y) - \frac{t}{\pi} \log\frac{1}{|y|} & \text{when} & y \ne 0, \\
\lim_{y \to 0} \alpha(t, y) - \frac{t}{\pi} \log\frac{1}{|y|} & \text{when} & y = 0 .
\end{array}  \right.
\]
It was shown by Le Gall \cite{LeG1985} that the limit above exists both almost surely and in $L^2$.
Below we will give an alternative proof of it, together with a Tanaka--Rosen--Yor formula for
$\gamma$, based on Theorem \ref{th:spec_TRY}. Yor \cite[(2.k)]{Yor1985} proved convergence in
probability by a similar approach.

Prior to that, let us give the expectation of $\gamma$, cf. \cite{LeG1985}.
\begin{cor} \label{co:egamma}
\begin{equation}\label{eq:egamma}
\mathbb{E} \gamma(t,y) = \left\{
\begin{array}{lll}
\frac{t}{\pi} \log|y| - \frac{|y|^2+2t}{4 \pi} \textup{Ei}\left(-\frac{|y|^2}{2t}\right)
- \frac{t}{2 \pi} e^{-\frac{|y|^2}{2t}} & \text{when} & y \ne 0, \\
\frac{t}{2 \pi} (\log(2t) - C - 1) & \text{when} & y = 0 ,
\end{array}  \right.
\end{equation}
where \textup{Ei} denotes the exponential integral function and $C$ is Euler's constant. Thus this
expectation is finite and continuous for every $(t,y) \in \mathbb{R}_+\times\mathbb{R}^2$.
\end{cor}
\begin{proof}
In formula (\ref{eq:TRYmain}) the expectation of the left side is given by Lemma \ref{le:Xprop}(a),
while the expectation of the stochastic integral term is 0 by Lemma \ref{le:contversion}(a), see the
Appendix.
\end{proof}

Since $\gamma(t,y) $ has finite expectation for any $(t,y)$, quite often in the literature the
renormalized self-intersection local time is defined by subtracting its expected value. Since that
would complicate some formulae below, here we do not follow that practice.

\begin{thm} \label{th:TRYmain0}
Combine the first and the last terms on the right side of (\ref{eq:TRYmain}) into a $\gamma$ term.
When $y \to 0$, the resulting terms in (\ref{eq:TRYmain}) converge almost surely and in $L^2$ to
\begin{equation}\label{eq:TRYmain0}
\int_{0}^{t} \log |W(t) - W(u)| \di u = \int_{0}^{t} \int_{0}^{v} \frac{W(v)-W(u)}{|W(v)-W(u)|^2} \di
u \cdot \di W(v) + \pi  \; \gamma(t, 0) .
\end{equation}
Moreover, $\gamma (t, y)$ is finite-valued and continuous in $(t,y)$ even when $y=0$.
\end{thm}
\begin{proof}
Let us denote the stochastic integral in (\ref{eq:TRYmain}) by $Y(t,y)$, cf. (\ref{eq:Yterm}) in the
Appendix. Define $Y_n(t):=Y(t,(n^{-1},0))$, $n \ge 1$. By Lemma \ref{le:contversion}(a), for each $n$,
$Y_n(t)$ is a continuous $L^2$-martingale. By (\ref{eq:KolmChenex}), for any $k \ge 1$,
\begin{equation}\label{eq:Y2k}
\mathbb{E}|Y_{2^{k}}(t) - Y_{2^{k-1}}(t)|^3 \le c_0(K) k^5 2^{-3k} \qquad (0 \le t \le K),
\end{equation}
where $c_0(K)$ is a constant depending only on $K$.

Thus a basic martingale inequality implies for any $k \ge 1$ that
\begin{multline*}
\mathbb{P}\left(\sup_{0 \le t \le K} |Y_{2^{k}}(t) - Y_{2^{k-1}}(t)| \ge 2^{-k/2}\right) \le 2^{3k/2}
\mathbb{E}|Y_{2^{k}}(K) - Y_{2^{k-1}}(K)|^3 \\
\le c_0(K) k^5 2^{-3k/2} \le c_1(K) 2^{-k} ,
\end{multline*}
where $c_1(K)$ is a constant depending on $K$.

An application of the Borel--Cantelli lemma then yields that $Y_{2^k}(t)$ almost surely uniformly
converges for $t \in[0,K]$ to a continuous $L^2$-martingale as $k \to \infty$. (The convergence is
also in $L^2$ by (\ref{eq:Y2k}).) Using isometry, the expression
\[
\int_{0}^{v} \frac{W(v)-W(u)-(2^{-k},0)}{|W(v)-W(u)-(2^{-k},0)|^2} \di u
\]
converges in $L^2$ as well when $k \to \infty$. Since for any fixed $v$, $W(v)-W(u)$ almost surely
does not equal to 0 when $0 \le u < v$, here the integrand can be dominated for any large enough $k
\ge k_0(\omega)$. So the integral and the limit can be interchanged and we get that the limit of
$Y_{2^k}$ is the continuous $L^2$-martingale $Y(t,0)$, which is the first term on the right side of
(\ref{eq:TRYmain0}).

Moreover, $Y(t,y)$ converges to $Y(t,0)$ as well when $y \to 0$, a.s. uniformly for $t \in [0,K]$.
For, by an argument based on (\ref{eq:KolmChenex}), similar to the above one, given any $\epsilon >
0$, for any $y$, $2^{-k-1} \le |y| \le 2^{-k}$ with $k$ large enough,
\[
|Y(t, y) - Y(t, 0)| \le |Y(t,y) - Y_{2^{k-1}}(t)| + |Y_{2^{k-1}}(t) - Y(t,0)| < \epsilon
\]
a.s. uniformly for $t \in [0,K]$.

The convergence as $y \to 0$ of the term $X(t,y)$ on the left side of (\ref{eq:TRYmain}) can be
treated analogously by Lemma \ref{le:Xprop}(b) of the Appendix.
\end{proof}

The following \emph{occupation time formulae}, cf. \cite{LeG1985}, follow from the previous results.
\begin{cor} \label{co:occupation}
Suppose that $f: \mathbb{R}^2 \to \mathbb{R}$ is a bounded, Borel measurable function. Then
\begin{multline}\label{eq:occup1}
\int_{0}^{t} \int_{0}^{v} f(W(v) - W(u)) \di u \di v = \int_{\mathbb{R}^2} f(x) \alpha (t,x) \di x \\
= \int_{\mathbb{R}^2} f(x) \left\{\gamma (t,x) - \frac{t}{\pi} \log|x|\right\} \di x.
\end{multline}
Alternatively,
\begin{multline}\label{eq:occup2}
\int_{0}^{t} \int_{0}^{v} \left\{f(W(v) - W(u)) - \mathbb{E}f(W(v) - W(u)) \right\}\di u \di v \\
= \int_{\mathbb{R}^2} f(x) \left\{\gamma (t,x) - \mathbb{E}\gamma (t,x)\right\} \di x .
\end{multline}

\end{cor}
\begin{proof}
It is enough to show (\ref{eq:occup1}) for indicator functions of discs; by standard methods, linear
combinations of indicators extend to a general $f$. So let us take a closed disc $B_r(a)$ and show
that
\begin{multline}\label{eq:occup}
\int_{0}^{t} \int_{0}^{v} \mathbf{1}_{B_r(a)}(W(v) - W(u)) \di u \di v = \int_{B_r(a)} \alpha (t,y)
\di y \\
= \int_{B_r(a)} \gamma (t,y) \di y - \frac{t}{\pi} \int_{B_r(a)} \log|y| \di y .
\end{multline}

From the results above we know that the stochastic integral term on the right side of
(\ref{eq:TRYmain}) has a version $Y(t,z)$ which is continuous for any $(t,y) \in
[0,K]\times\mathbb{R}^2$. (\ref{eq:secondterm}) gives
\[
\int_{0}^{t} \int_{0}^{v} (\nabla \phi^{\delta })\left(W(v)-W(u)-y\right) \di u \cdot \di W(v) =
\frac{1}{\pi \delta ^2} \int_{B_{\delta }(y)} Y(t,z) \di z .
\]
Substitute this into (\ref{eq:spec_TRY}):
\begin{multline*}
\int_{0}^{t} \phi^{\delta }\left(W(t) - W(u) - y \right) \di u - t \phi^{\delta }(y)
- \frac{1}{\pi \delta ^2} \int_{B_{\delta }(y)} Y(t,z) \di z \\
= \frac{1}{\delta ^2} \; \int_{0}^{t} \int_{0}^{v} \mathbf{1}_{B_{\delta }(y)}(W(v)-W(u)) \di u \di v
.
\end{multline*}
Then integrate this equality over a closed disc $B_r(a)$ with respect to $y$:
\begin{multline} \label{eq:occupdel}
\int_{B_r(a)} \left\{ \int_{0}^{t} \phi^{\delta }\left(W(t) - W(u) - y \right) \di u - t \phi^{\delta
}(y)
- \frac{1}{\pi \delta ^2} \int_{B_{\delta }(y)} Y(t,z) \di z \right\} \di y \\
= \frac{1}{\delta ^2} \; \int_{0}^{t} \int_{0}^{v} \int_{B_r(a)} \mathbf{1}_{B_{\delta
}(y)}(W(v)-W(u)) \di y \di u \di v .
\end{multline}

Now, it is clear that
\begin{equation}\label{eq:discs}
\lim_{\delta \to 0^+} \frac{1}{\pi \delta ^2} \int_{B_r(a)} \mathbf{1}_{B_{\delta }(y)}(z) \di y =
\mathbf{1}_{B_{r }^{\circ}(a)}(z) + \frac12 \mathbf{1}_{\partial B_{r }(a)}(z) ,
\end{equation}
where $B_{r }^{\circ}(a) = \{x: |x-y| < r \}$. Take limit in (\ref{eq:occupdel}) as $\delta \to 0^+$.
By Theorem \ref{th:spec_TRY}, the limit of the expression in the braces is $\pi \alpha(t,y)$ when $y
\ne 0$. Since $\alpha(t,y) = \gamma(t,y) - \frac{t}{\pi} \log|y|$, and $\gamma$ is a.s. continuous
everywhere while $\int_{B_r(a)} \log|y| \di y$ is finite even if $0 \in B_r(a)$, this gives us the
right hand sides of (\ref{eq:occup}).

The left hand side of (\ref{eq:occup}) is obtained as the limit of the right hand side of
(\ref{eq:occupdel}), using (\ref{eq:discs}), since it has zero probability that $W(v)-W(u) \in
\partial B_{r }(a)$. This completes the proof of (\ref{eq:occup1}).

Now take expectation in (\ref{eq:occup1}):
\[
\int_{0}^{t} \int_{0}^{v} \mathbb{E}f(W(v) - W(u)) \di u \di v = \int_{\mathbb{R}^2} f(x)
\left\{\mathbb{E}\gamma (t,x) - \frac{t}{\pi} \log|x|\right\} \di x.
\]
Subtract this from (\ref{eq:occup1}), and the result is (\ref{eq:occup2}).

\end{proof}

Finally, we add some remarks.
We can combine the occupation time formula (\ref{eq:occup1}) and formula (\ref{eq:TRY}), with
$y=0$, when $\phi$ is a $C^3$ scalar field:
\begin{multline} \label{eq:TRYoccup}
\int_{0}^{t} \phi\left(W(t) - W(u) \right) \di u
- \int_{0}^{t} \int_{0}^{v} (\nabla \phi)\left(W(v) - W(u)\right) \di u \cdot \di W(v) \\
= \frac12 \int_{\mathbb{R}^2} (\Delta \phi)(x) \alpha(t,x) \di x + t \phi(0) \\
= \frac12 \int_{\mathbb{R}^2} (\Delta \phi)(x) \left\{\gamma(t,x) - \frac{t}{\pi} \log|x|\right\} \di
x + t \phi(0).
\end{multline}
Green's theorem implies that $\phi(0) = \frac{1}{2 \pi} \int_{\mathbb{R}^2} (\Delta \phi)(x) \log|x|
\di x$ holds when $\phi(x) - x \cdot (\nabla \phi)(x) \, \log|x|$ tends to 0 as $|x| \to \infty$. In
this case (\ref{eq:TRYoccup}) simplifies to the following Tanaka--Rosen--Yor formula:
\begin{multline*}
\int_{0}^{t} \phi\left(W(t) - W(u) \right) \di u
= \int_{0}^{t} \int_{0}^{v} (\nabla \phi)\left(W(v) - W(u)\right) \di u \cdot \di W(v) \\
+ \frac12 \int_{\mathbb{R}^2} (\Delta \phi)(x) \gamma(t,x) \di x .
\end{multline*}

Further, comparing the first equality of (\ref{eq:TRYoccup}) with (\ref{eq:TRYlast}) results the
almost sure limit
\begin{equation}\label{eq:cont_disc}
\lim_{m \to \infty} \sum_{x \in 2^{-m}\mathbb{Z}^2} \alpha _m(t, x) \: (\Delta \phi)(x) \: 2^{-2m}
= \int_{\mathbb{R}^2}  \alpha(t,x) \:(\Delta \phi)(x) \di x .
\end{equation}
Now, almost surely, the support of the continuous $\alpha$, and then by (\ref{eq:strong_d}),
the support of each $\alpha_m$
for $m \ge m_0(\omega)$, can be covered by a finite disc $B_R(0)$ with a large enough radius
$R=R(\omega)$. Thus Lemma \ref{le:conv} implies that the sum here can be replaced by an
integral. Also, the Poisson equation $\Delta \phi = f$ can be solved in the plane for
any continuous $f$, $f(x) = O(|x|^{-2-\delta})$, $\delta > 0$, by the formula
$\phi(y) = \frac{1}{2 \pi} \int_{\mathbb{R}^2} (\Delta \phi)(x) \log|x-y| \di x$. Thus, for
any such $f$, (\ref{eq:cont_disc}) can be written as
\begin{equation}\label{eq:alpha_conv}
\lim_{m \to \infty} \int_{\mathbb{R}^2} \alpha _m(t, x) \: f(x) \di x
= \int_{\mathbb{R}^2}  \alpha(t,x) \: f(x) \di x   \qquad \text{a.s.}
\end{equation}
This weak convergence formula supplements the basic definition (\ref{eq:SILTlim}) of self-intersection
local time.

%%%%%%%%%%%%%%%%%%%%%%%%%%%%%%%%%%%%%%%%%%%%%%%%%%%%%%%%%%%%%%%%%%%

\section{Appendix: Some technical lemmas}
\label{sec:tech_lem}

This lemma says that the sum in (\ref{eq:speclast}) can be approximated by an integral.
\begin{lem} \label{le:conv}
Almost surely, for any $t \in \mathbb{R}_+$, $y \in \mathbb{R}^2$ and $\delta > 0$ fixed,
\begin{multline*}
\sum_{ x \in 2^{-m}\mathbb{Z}^2 \cap B_{\delta} (y)} \alpha _m(t, x) \: 2^{-2m} = \int_{B_{\delta}
(y)} \alpha _{m}(t, x)  \di x + \delta \: O(m^2 2^{-m}) .
\end{multline*}
\end{lem}
\begin{proof}
First, let us estimate the error between the sum of a discrete function and the integral of an
interpolated function over a rectangular domain $A_n = [x_0, x_n) \times [y_0, y_n)$ with vertices on
a grid $h \mathbb{Z}^2$, in general. So let $f : h \mathbb{Z}^2 \to \mathbb{R}$ be a discrete function
and $f(x,y)$, $(x,y) \in \mathbb{R}^2$, be obtained from $f$ by linear interpolation on triangles of
the grid, as it was described for $\alpha _h$ in (\ref{eq:int_pol}). Put
\[
S_n := h^2 \sum_{(x_i, y_j) \in A_n} f(x_i,y_j) = h^2 \sum_{i=1}^{n} \sum_{j=1}^{n} f(x_{i-1},y_{j-1})
,
\]
and
\begin{multline*}
T_n := \int_{A_n} f(x,y) \di x \di y \\
= \frac{h^2}{6} \sum_{i=1}^{n} \sum_{j=1}^{n} \left\{f(x_{i-1},y_{j-1}) + 2 f(x_i, y_{j-1}) + 2
f(x_{i-1}, y_j) + f(x_i,y_j)\right\} ,
\end{multline*}
cf. (\ref{eq:int_pol_int}). Then in the error $T_n - S_n$, all contributions of inner vertices cancel
and only the contribution of vertices at the boundary of $A_n$ remain:
\begin{multline*}
T_n - S_n = \frac{h^2}{2} \sum_{i=1}^{n-1} \left\{f(x_i,y_n) +  f(x_n, y_i) -  f(x_0, y_i)
- f(x_i,y_0)\right\} \\
 + \frac{h^2}{6} \left\{ 2 f(x_0,y_n) +  2f(x_n, y_0) + f(x_n, y_n) - 5f(x_0,y_0) \right\}.
\end{multline*}

Returning to the statement of the lemma, by our definition in Section \ref{sec:DSILT}, $\alpha _{m}(t,
x)$ is obtained by linear interpolation on triangles of the grid with mesh $h=2^{-m}$. If one
considers the difference of the sum and the integral over a disc $B_{\delta }(y)$, the cancelation of
inner vertices still holds, and only the contributions of vertices adjacent to the circumference
remain. The number of these latter vertices is of the order of $\delta \; O(2^{m})$. Thus by
(\ref{eq:ordSILT}) we have that
\begin{multline*}
\left| \sum_{x \in 2^{-m}\mathbb{Z}^2 \cap B_{\delta} (y)} \alpha _m(t, x) \: 2^{-2m}
- \int_{B_{\delta} (y)} \alpha _{m}(t, x)  \di x \right| \\
\le \delta \; O(2^{m}) \; O(m^2) \; 2^{-2m} = \delta \; O(m^2 2^{-m}) .
\end{multline*}
This completes the proof of the lemma.
\end{proof}

We need the following representation of $\nabla \phi^{\delta }$.
\begin{lem} \label{le:gauss}
\begin{multline}\label{eq:nabla_phi}
(\nabla \phi^{\delta })(x) = \frac{1}{\pi \delta ^2} \int_{B_{\delta }(0)} \frac{x - z}{|x - z|^2}
\: \di z \\
= -\frac{1}{\pi \delta ^2} \int_{C_{\delta }(0)} \log|z - x| \; n_0(z) \: \di s(z)
 \qquad (x \in \mathbb{R}^2) ,
\end{multline}
where $B_{\delta }(0)$ is the closed disc centered at the origin with radius $\delta > 0$, $C_{\delta
}(0)$ is its counterclockwise directed boundary, $n_0$ is the outward unit normal along the boundary,
and $\di s$ denotes integration with respect to arc length.
\end{lem}
\begin{proof}
The second equality follows from a standard theorem of vector analysis; note that the discontinuity of
the integrand in the second term when $|x| \le \delta $ is not essential. To prove that the first term
equals the third, because of rotational symmetry, it is enough to consider points $x = (-a,0)$, $a \ge
0$. Then the third term of (\ref{eq:nabla_phi}) becomes
\[
-\frac{1}{2 \pi \delta } \int_{0}^{2 \pi} \log \left((a + \delta \cos \theta )^2 + (\delta \sin
\theta)^2 \right) \left(\cos \theta , \sin \theta \right) \di \theta = \left(-\frac{1}{\delta } \Psi
\left( \frac{a}{\delta }\right), 0 \right),
\]
where
\[
\Psi (u) := \frac{1}{2 \pi} \int_{0}^{2 \pi} \log \left(u^2 + 1 + 2 u \cos \theta \right) \cos \theta
\di \theta = \left\{\begin{array}{ll}
           u & \text{for } 0 \le u \le 1, \\
           \frac{1}{u} & \text{for } u \ge 1 .
         \end{array} \right.
\]
These prove the equality with $\nabla \phi^{\delta }$ given by (\ref{eq:nablaphi}).
\end{proof}

The next lemma establishes some important properties of the stochastic integral appearing in Theorem
\ref{th:spec_TRY}.
\begin{lem} \label{le:contversion}
Fix an arbitrary $K>0$. Consider the stochastic integral
\begin{equation}\label{eq:Yterm}
Y(t,y) := \int_{0}^{t} \int_{0}^{v} \frac{W(v)-W(u)-y}{|W(v)-W(u)-y|^2} \di u \cdot \di W(v) .
\end{equation}
Then the following properties hold.
\begin{enumerate}[(a)]
\item $Y(t,y)$ is a continuous $L^2$-martingale with expectation 0 as a function of $t \in [0, K]$ for
any fixed $y \in \mathbb{R}^2$.

\item
\begin{equation}\label{eq:KolmChen}
\mathbb{E}|Y(t,y) - Y(t,y')|^3 \le C |y - y'|^{2+\beta}
\end{equation}
with a finite $C=C(K,a)$ and with an arbitrary $\beta \in (0,1)$ for any $t \in [0, K]$ and $|y|, |y'|
\ge a$, where $a>0$ is arbitrary, fixed. More exactly,
\begin{equation}\label{eq:KolmChenex}
\mathbb{E}|Y(t,y) - Y(t,y')|^3 \le c(K) \left(1 + \log_+ 2K a^{-2} \right) \log^4 \frac{1}{|y-y'|}  |y
- y'|^{3},
\end{equation}
where $c(K)$ is a finite constant depending only on $K$ and $\log_+ x := 0 \vee \log x$.

\item $Y(t,y)$ has a version which is a.s. a continuous function of $y \ne 0$. In fact, it has a
version which is a.s. a continuous function of $(t,y)$ when $y \ne 0$.
\end{enumerate}

\end{lem}
\begin{proof}

\begin{enumerate}[(a)]
\item Since for any fixed $v$, $\tilde{W}(u) := W(v)-W(v-u)$ is planar Brownian motion as well that
starts from $0$, we have
\begin{multline*}
\mathbb{E}|Y(t,y)|^2 = \int_0^t \mathbb{E}\left| \int_0^v  \frac{W(v)-W(u)-y}{|W(v)-W(u)-y|^2}
\di u \right|^2 \di v \\
\le \int_0^t \mathbb{E}\left( \int_0^t \frac{1}{|\tilde{W}(u) - y|} \di u \right)^2 \di v = t
\mathbb{E}\left( \int_0^t \frac{1}{|W(u) - y|} \di u \right)^2 .
\end{multline*}

Then by symmetry and by the independence of increments of $W$, we get that
\begin{multline*}
\mathbb{E}|Y(t,y)|^2 \le t  \mathbb{E}\left( \int_0^t \di u_1 \int_0^t \di u_2 \frac{1}
{|W(u_1) - y||W(u_2) - y|} \right) \\
= 2t  \int_{[0,t] \times \mathbb{R}^2} \di u_1 \di z_1 \frac{e^{-\frac{|z_1|^2}{2 u_1}}}{2 \pi u_1
|z_1 - y|} \int_{[u_1,t] \times \mathbb{R}^2} \di u_2 \di z_2 \frac{e^{-\frac{|z_2-z_1|^2}{2
(u_2-u_1)}}} {2 \pi (u_2-u_1)|z_2 - y|} .
\end{multline*}

Writing $z_2 = y + r (\cos \theta , \sin \theta )$, $z_1 = y + \rho (\cos \alpha , \sin \alpha)$ and
$v=u_2 - u_1$, for the inner integral here we obtain
\begin{multline*}
 \int_0^{t-u_1} \frac{\di v}{v} \int_0^{2\pi} \frac{\di \theta}{2\pi} \int_0^{\infty} \di r
 \, e^{-\frac{r^2+\rho^2-2r\rho\cos(\theta-\alpha)}{2v}} \\
\le \sqrt{2\pi} \int_0^{t-u_1} \frac{\di v}{\sqrt{v}} \int_0^{2\pi} \frac{\di \theta}{2\pi}
\int_0^{\infty} \frac{\di r}{\sqrt{2 \pi v}}  e^{-\frac{(r-\rho)^2}{2v}} \le 2 \sqrt{2 \pi (t-u_1)} .
\end{multline*}

Thus
\begin{multline*}
\mathbb{E}|Y(t,y)|^2 \le 2t \int_0^{t} \frac{\di u_1}{u_1} \int_0^{2\pi} \frac{\di \alpha}{2\pi}
\int_0^{\infty} \di \rho \, e^{-\frac{\rho^2}{2u_1}} 2 \sqrt{2 \pi (t-u_1)} \\
\le 8\pi K^2 < \infty \qquad (0 \le t \le K, y \in \mathbb{R}^2).
\end{multline*}

\item First, by the Burkholder--Davis--Gundy inequality, for any $m > 0$, there exists a finite $c_m$
such that
\begin{multline*}
\mathbb{E}|Y(t,y) - Y(t,y')|^m \\
\le c_m \mathbb{E} \left(\int_{0}^{t} \left| \int_{0}^{v} \frac{W(v)-W(u)-y}{|W(v)-W(u)-y|^2} -
\frac{W(v)-W(u)-y'}{|W(v)-W(u)-y'|^2} \di u \right|^2 \di v \right)^{\frac{m}{2}} .
\end{multline*}
Using the elementary vector equality $\left|a |a|^{-2} - b |b|^{-2}\right| = |a-b| (|a| |b|)^{-1}$,
and the fact that for any fixed $v$, $\tilde{W}(u) := W(v)-W(v-u)$ is planar Brownian motion as well
that starts from $0$, it follows that
\begin{multline} \label{eq:mthmoment}
\mathbb{E}|Y(t,y) - Y(t,y')|^m \\
\le c_m |y - y'|^m \mathbb{E} \left(\int_{0}^{t} \left( \int_{0}^{t} \frac{1}{|\tilde{W}(u)-y|
|\tilde{W}(u)-y'|} \di u \right)^2 \di v \right)^{\frac{m}{2}} \\
= c_m t^{\frac{m}{2}} |y - y'|^m \mathbb{E} \left(\int_{0}^{t}  \frac{1}{|W(u)-y| |W(u)-y'|} \di u
\right)^{m},
\end{multline}
where for any fixed $v$, $\tilde{W}(u) := W(v)-W(v-u)$ is planar Brownian motion as well that starts
from $0$.

Thus to show (\ref{eq:KolmChen}), it is enough to give a suitable upper estimate for the last
expectation in (\ref{eq:mthmoment}) when the norms of $y$ and $y'$ are bounded below by $a$ and $t \in
[0, K]$. Now, by symmetry and by the independence of increments of Brownian motion, with $m=3$ we
obtain that
\begin{multline} \label{eq:thirdmom}
\mathbb{E} \left(\int_{0}^{t}  \frac{1}{|W(u)-y| |W(u)-y'|} \di u  \right)^{3} \\
= 6 \mathbb{E} \int_{0}^{t} \di u_1 \int_{u_1}^{t} \di u_2 \int_{u_2}^{t} \di u_3 \prod_{j=1}^{3}
\frac{1}{|W(u_j)-y| |W(u_j)-y'|}  \\
= 6   \int_{[0,t] \times \mathbb{R}^2} \di u_1 \di z_1 \frac{e^{-\frac{|z_1|^2}{2 u_1}}}{2 \pi u_1}
\int_{[u_1,t] \times \mathbb{R}^2} \di u_2 \di z_2 \frac{e^{-\frac{|z_2-z_1|^2}{2 (u_2-u_1)}}}
{2 \pi (u_2-u_1)} \\
 \times \int_{[u_2,t] \times \mathbb{R}^2} \di u_3 \di z_3 \frac{e^{-\frac{|z_3-z_2|^2}{2 (u_3-u_2)}}}
 {2 \pi (u_3-u_2)} \prod_{j=1}^{3} \frac{1}{|z_j-y| |z_j-y'|} .
\end{multline}

Without loss of generality, from now on we may assume that $0 < |y-y'| \le 1/2$. $B_r(x)$ will denote
the closed disc centered at $x$ with radius $r$. Here and later we use the following covering:
\[
\mathbb{R}^2 = \left(B_1(y) \cup B_1(y')\right)^c \cup \bigcup_{n=-1}^{N} \left(C_n(y) \cup
C_n(y')\right),
\]
where $C_n(y):= B_{2^{n}|y-y'|}(y) \cap B_{2^{n-1}|y-y'|}^c(y) \cap B_{2^{n-1}|y-y'|}^c(y')$, $0 \le n
\le N$, $N=\lceil \log(|y-y|^{-1} / \log 2 \rceil$, and $C_{-1}(y):=B_{2^{-1}|y-y'|}(y)$. For
$C_n(y')$ the definitions are similar.

To show the method, let us estimate the innermost integral $I_1$ in (\ref{eq:thirdmom}) using the
above covering of $\mathbb{R}^2$. First,
\begin{multline*}
I_{1}^{*} := \int_{u_2}^t \di u_3 \int_{\left(B_1(y) \cup B_1(y')\right)^c} \di z_3
\frac{e^{-\frac{|z_3-z_2|^2}{2 (u_3-u_2)}}}{2 \pi (u_3-u_2)} \frac{1}{|z_3-y| |z_3-y'|} \\
\le \int_{0}^{t-u_2} \di v \int_{\mathbb{R}^2} \di z_3 \frac{e^{-\frac{|z_3-z_2|^2}{2 v}}}{2 \pi v} =
t-u_2.
\end{multline*}

Second, write $z_3 = y + r (\cos \theta , \sin \theta )$, $z_2 = y + \rho (\cos \alpha , \sin \alpha
)$, and for $n=0,1, \dots, N$ obtain that
\begin{multline*}
I_{1,n}(y) := \int_{u_2}^t \di u_3 \int_{C_n} \di z_3 \frac{e^{-\frac{|z_3-z_2|^2}{2 (u_3-u_2)}}}
{2 \pi (u_3-u_2)} \frac{1}{|z_3-y| |z_3-y'|} \\
\le  \int_{0}^{2\pi} \frac{\di \theta}{2\pi} \int_{0}^{t-u_2} \frac{\di v}{v}
\int_{2^{n-1}|y-y'|}^{2^{n}|y-y'|} \di r \frac{e^{-\frac{r^2 + \rho^2 - 2 r
\rho \cos(\theta - \alpha )}{2 v}}}{2^{n-1}|y-y'|} \\
= \int_{0}^{2\pi} \frac{\di \theta}{2\pi} \int_{0}^{t-u_2} \frac{\di v}{\sqrt{v}} e^{- \frac{\rho^2
\sin^2 \theta}{2v}} \int_{\left(2^{n-1}|y-y'| - \rho \cos \theta \right)/\sqrt{v}}^{\left(2^{n}|y-y'|
- \rho \cos \theta \right)/\sqrt{v}} \di s \frac{e^{-\frac{s^2}{2}}}{2^{n-1}|y-y'|} \\
\le \int_{0}^{2\pi} \frac{\di \theta}{2\pi} \int_{0}^{t-u_2} \frac{\di v}{v} e^{- \frac{\rho^2 \sin^2
\theta}{2v}} .
\end{multline*}
(The value of $\alpha$ clearly does not matter, so it was replaced by 0.) Here one can use the simple
estimate
\[
\int_{0}^{x} \frac{1}{v} e^{-\frac{b}{v}} \di v \le \frac{1}{e} + \log_+ \frac{x}{b}.
\]
Then
\begin{multline*}
I_{1,n}(y) \le \int_{0}^{2\pi} \left(\frac{1}{e} + \log_+ \frac{2(t-u_2)}{\rho^2
\sin^2 \theta} \right)  \frac{\di \theta}{2\pi} \\
\le \frac{1}{e} + \log_+ \frac{2(t-u_2)}{\rho^2} -  \int_{0}^{2\pi} \log(\sin^2 \theta )
\frac{\di \theta}{2\pi} \\
< 2 + \log_+ \frac{2(t-u_2)}{|z_2 - y|^2}
\end{multline*}
$(n = 0, 1, \dots, N)$, since $\rho = |z_2 - y|$. The estimate for $I_{1,-1}(y)$ is the same, and so
is for any $I_{1,n}(y')$ replacing $y$ by $y'$.

In sum, the estimate for the innermost integral is
\begin{multline}
I_1 \le I_{1}^{*} + \sum_{n=-1}^{N} \left(I_{1,n}(y) + I_{1,n}(y')\right) \\
\le t-u_2 + 12 \log \frac{1}{|y-y'|} \left(2 + \log_+ \frac{2(t-u_2)}{|z_2 - y| |z_2 - y'|} \right) ,
\end{multline}
since $N+2 < 6 \log(1/|y-y'|)$ when $0 < |y-y'| \le 1/2$.

The estimation of the second and third integrals in (\ref{eq:thirdmom}) can go in a similar fashion.
Omitting the details, the result is
\begin{multline} \label{eq:thirdmomest}
\mathbb{E} \left(\int_{0}^{t}  \frac{1}{|W(u)-y| |W(u)-y'|} \di u  \right)^{3} \\
\le c(K) \log^4 \frac{1}{|y-y'|} \left(1 + \log_+ \frac{2t}{|y| |y'|} \right) \\
\le c(K) \log^4 \frac{1}{|y-y'|} \left(1 + \log_+ 2K a^{-2} \right),
\end{multline}
for any $y,y'$ such $|y|, |y'| \ge a$, $0 < |y-y'| \le 1/2$, and $t \in [0,K]$, where $c(K)$ is a
finite constant depending on $K$. By (\ref{eq:mthmoment}) this verifies (\ref{eq:KolmChen}) with
arbitrary $\beta < 1$, and so proves the lemma.

\item By (b), $Y(t,y)$ as a function of $y$ satisfies the condition of a special case of the
Kolmogorov--Chentsov theorem, so a.s. it has a continuous version as a function of $y$ when $y \ne 0$.

One can similarly show that $Y(t,y)$ has a version which is a continuous function of $(t,y)$ when $y
\ne 0$, e.g. considering fourth moment instead of the third.
\end{enumerate}
\end{proof}

This last lemma investigates the properties of the integral on the left side of formula
(\ref{eq:TRYmain}).
\begin{lem} \label{le:Xprop}
Consider the integral
\[
X(t,y) := \int_{0}^{t} \log |W(t) - W(u) - y| \di u  \qquad (t \ge 0, y \in \mathbb{R}^2).
\]
Then it has the following properties.
\begin{enumerate}[(a)]
\item
\begin{equation}\label{eq:Xprop1}
\mathbb{E} X(t,y) = t \log|y| - \frac{|y|^2+2t}{4} \textup{Ei}\left(-\frac{|y|^2}{2t}\right) - \frac12
t e^{-\frac{|y|^2}{2t}} \quad (y \ne 0),
\end{equation}
\begin{equation}\label{eq:Xprop2}
\lim_{y \to 0} \mathbb{E} X(t,y) = \mathbb{E} X(t,0) = \frac{t}{2} (\log(2t) - C - 1),
\end{equation}
where \textup{Ei} denotes the exponential integral function and $C$ is Euler's constant.

\item
\begin{equation}\label{eq:XKolmChen}
\mathbb{E}|X(t,y) - X(t,y')|^3 \le C |y - y'|^{2+\beta}
\end{equation}
with a finite $C=C(K,a)$ and with an arbitrary $\beta \in (0,1)$ for any $t \in [0, K]$ and $|y|, |y'|
\ge a$, where $a>0$ is arbitrary, fixed. More exactly,
\begin{equation}\label{eq:XKolmChenex}
\mathbb{E}|X(t,y) - X(t,y')|^3 \le c(K) \left(1 + \log_+ 2K a^{-2} \right) \log^4 \frac{1}{|y-y'|}  |y
- y'|^{3},
\end{equation}
where $c(K)$ is a finite constant depending only on $K$.

\end{enumerate}
\end{lem}
\begin{proof}
\begin{enumerate}[(a)]
\item For any fixed $t$, $\tilde{W}(u) := W(t)-W(t-u)$ is planar Brownian motion as well that starts
from $0$. Thus
\begin{multline*}
\mathbb{E} X(t,y) = \int_{0}^{t} \log |\tilde{W}(u) - y| \di u = \int_0^t \int_{\mathbb{R}^2}
\frac{\log|x-y|}{2\pi u} e^{-\frac{|x|^2}{2u}} \di x \di u \\
= \int_0^t \di u \int_0^{\infty} \di r \frac{r}{u} e^{-\frac{r^2}{2u}} \int_0^{2 \pi} \frac{\di
\theta}{2 \pi} \frac12 \log\left(r^2 + \rho^2 -2r\rho \cos(\theta-\alpha) \right) ,
\end{multline*}
where $x=r(\cos\theta, \sin\theta)$ and $y=\rho(\cos\alpha, \sin\alpha)$. It is clear that the last
integral does not depend on $\alpha$, so we can replace $\alpha$ by 0. Since
\[
\int_0^{2 \pi} \frac{\di \theta}{2 \pi} \frac12 \log\left(r^2 + \rho^2 -2r\rho \cos(\theta) \right) =
\log(r \vee \rho),
\]
it follows that
\[
\mathbb{E} X(t,y) = \int_{0}^{t} \di u \int_0^{\infty} \di r \frac{r}{u} e^{-\frac{r^2}{2u}} \log(r
\vee \rho) ,
\]
that gives exactly the results (\ref{eq:Xprop1}) and (\ref{eq:Xprop2}).

\item
\begin{multline*}
\mathbb{E}|X(t,y) - X(t,y')|^3 = \mathbb{E} \left| \int_0^t \log \frac{|\tilde{W}(u) - y|}
{|\tilde{W}(u) - y'|} \di u \right|^3 \\
\le \mathbb{E} \left( \int_0^t \left|\log \frac{|\tilde{W}(u) - y|}{|\tilde{W}(u) - y'|} \right| \di u
\right)^3 .
\end{multline*}

Using the inequality $|\log b - \log a| \le (a \wedge b)^{-1} |b-a| \le (a^{-1} + b^{-1}) |b-a|$ for
$a,b > 0$, and then symmetry and the independent increments of $W$, we obtain that
\begin{multline*}
\mathbb{E}|X(t,y) - X(t,y')|^3 \\
 \le |y - y'|^3 \; \mathbb{E} \left( \int_0^t \left( |\tilde{W}(u) - y|^{-1}
 + |\tilde{W}(u) - y'|^{-1} \right) \di u \right)^3 \\
\le 6 |y - y'|^3 \; \int_{[0,t] \times \mathbb{R}^2} \di u_1 \di z_1 \frac{e^{-\frac{|z_1|^2}{2
u_1}}}{2 \pi u_1} \int_{[u_1,t] \times \mathbb{R}^2} \di u_2 \di z_2 \frac{e^{-\frac{|z_2-z_1|^2}{2
(u_2-u_1)}}}
{2 \pi (u_2-u_1)} \\
 \times \int_{[u_2,t] \times \mathbb{R}^2} \di u_3 \di z_3 \frac{e^{-\frac{|z_3-z_2|^2}{2 (u_3-u_2)}}}
 {2 \pi (u_3-u_2)} \prod_{j=1}^{3} \left( \frac{1}{|z_j-y|} + \frac{1}{|z_j-y'|} \right).
\end{multline*}

Since this last formula is very similar to formula (\ref{eq:thirdmom}), the remaining part of the
proof is essentially the same, so omitted. The result differs from the one of Lemma
\ref{le:contversion}(b) only by a constant multiplier depending on $K$.

\end{enumerate}

\end{proof}

%%%%%%%%%%%%%%%%%%%%%%%%%%%%%%%%%%%%%%%%%%%%%%%%%%%%%%%%%%%%%%%%%%%

\end{document}